\documentclass[12pt,twoside,english,american]{scrartcl}
\usepackage{helvet}
\usepackage{url}
\usepackage[T1]{fontenc}
\usepackage[latin9]{inputenc}
\usepackage{amsmath}
\usepackage{scrpage2}
\usepackage{setspace}
\usepackage{amssymb}
\usepackage{enumerate}
\usepackage{mathtools}
\allowdisplaybreaks

\makeatletter

 \usepackage{amsthm}
 \theoremstyle{plain}
\newtheorem{thm}{Theorem}[section]
  \theoremstyle{plain}
  \newtheorem*{thm*}{Theorem}
  \theoremstyle{plain}
  
  \theoremstyle{remark}
  \newtheorem{rem}[thm]{Remark}
  \theoremstyle{definition}
  
 \theoremstyle{definition}
 \newtheorem*{defn*}{Definition}
  \theoremstyle{plain}
  
 \theoremstyle{definition}
  
  \theoremstyle{remark}

\usepackage[T1]{fontenc}    

\usepackage{tu-preprint}
\usepackage{amsmath}

\usepackage{euscript}
\usepackage{mathabx}
\allowdisplaybreaks[1] 
\usepackage{amsfonts}
\newenvironment{keywords}{ \noindent\footnotesize\textbf{Keywords and phrases:}}{}

\newenvironment{class}{\noindent\footnotesize\textbf{Mathematics subject classification 2010:}}{}

\usepackage{color,tu-preprint}

\newcommand{\Abs}[1]{\left\lVert#1\right\rVert}

\renewcommand*{\div}{\operatorname{div}}
\newcommand*{\curl}{\operatorname{curl}}
\newcommand{\R}{\mathbb{R}}
\newcommand*{\Grad}{\operatorname{Grad}}

\newcommand*{\abs}[1]{\lvert#1\rvert}

\newcommand*{\grad}{\operatorname{grad}}

\renewcommand*{\i}{\mathrm{i}}
\newcommand{\s}[1]{\mathcal{#1}}

\DeclareMathAccent{\Circ}{\mathalpha}{operators}{"17}
\newcommand{\interior}[1]{\Circ{#1}}
\renewcommand{\Im}{\operatorname{\mathfrak{Im}}}
\renewcommand{\Re}{\operatorname{\mathfrak{Re}}}

\theoremstyle{plain}
\newtheorem{Sa}[subsection]{Theorem}
\newtheorem{Sa*}[section]{Theorem}
\newtheorem{Le}[subsection]{Lemma}
\newtheorem{Le*}[section]{Lemma}
\newtheorem{Fo}[subsection]{Corollary}
\newtheorem{Fo*}[subsection]{Corollary}
\newtheorem{Prop}[subsection]{Proposition}
\newtheorem{Prop*}[section]{Proposition}

\theoremstyle{definition}
\newtheorem*{Def}{Definition}
\newtheorem{Ass}[subsection]{Assumption}
\newtheorem{Bei}[subsection]{Example}

\theoremstyle{remark}

\newtheorem{rems}[subsection]{Remark}

 \numberwithin{equation}{section}

\DeclareMathOperator{\TextRe}{Re}
\DeclareMathOperator{\TextIm}{Im}
\renewcommand{\Re}{\TextRe}
\renewcommand{\Im}{\TextIm}

\newcommand{\N}{\mathbb{N}}
\newcommand{\Z}{\mathbb{Z}}

\newcommand{\C}{\mathbb{C}}

\newcommand{\eps}{\varepsilon}

\DeclareMathOperator{\1}{\chi}

\newcommand{\dd}{\ \mathrm{d}}

\DeclareMathOperator{\Diverg}{Div}
\DeclareMathOperator{\diverg}{div}

\newcommand{\tor}[2]{\stackrel{#1\to #2}{\longrightarrow}}

\newcommand{\ben}{\begin{enumerate}[(i)]}
\newcommand{\een}{\end{enumerate}}

\renewcommand{\hat}{\widehat}
\renewcommand{\tilde}{\widetilde}
\renewcommand*{\epsilon}{\varepsilon}
\renewcommand*{\rho}{\varrho}

\makeatother

\usepackage{babel}

\begin{document}
\selectlanguage{english}%
\institut{Institut f\"ur Analysis / Department of Mathematical Sciences}

\preprintnumber{MATH-AN-07-2012}

\preprinttitle{On the theory of homogenization of evolutionary equations in Hilbert spaces}

\author{Marcus Waurick}


\selectlanguage{american}%
\setcounter{section}{-1}


\title{On the homogenization of partial integro-differential-algebraic equations}

\author{ Marcus Waurick\\
 TU Dresden / University of Bath\\
 marcus.waurick@tu-dresden.de }
\maketitle
\begin{abstract}
We present a Hilbert space perspective to homogenization of standard linear evolutionary boundary value problems in mathematical physics and provide a unified treatment for (non-)periodic homogenization problems in thermodynamics, elasticity, electro-magnetism and coupled systems thereof. The approach permits the consideration of memory problems as well as differential-algebraic equations. We show that the limit equation is well-posed and causal. We rely on techniques from functional analysis and operator theory only.
\end{abstract}
\setcounter{section}{-1}


\begin{class} 74Q15 (Effective constitutive equations), 35B27 (Homogenization; equations in media with periodic structure), 35Q74 (PDEs in connection with mechanics of deformable solids), 35Q61 (Maxwell equations), 35Q79 (PDEs in connection with classical thermodynamics and heat transfer) \end{class}

\newpage
\tableofcontents{}

\section{Introduction}

In the mathematical treatment of physical phenomena of heterogeneous materials one is confronted with partial differential equations with variable coefficients. Sometimes one observes that there is a scale separation, that is, there is a ``large'' and a ``small'' scale and the heterogeneities occur on the ``small'' scale only. In consequence, the coefficients are highly oscillatory. Therefore, the computational effort for solving these equations is considerably high. That is why one seeks for a replacement for the partial differential equation with heterogeneous coefficients, such that this replacement is easier to solve with a computer. In some cases, it is possible to derive such a replacement by studying the limit behavior of the equation by letting the ratio of ``small'' scale over ``large'' scale tend to $0$. So, one asks whether the solutions corresponding to strictly positive ratio converge in some sense for the ratio tending to $0$. Given the convergence of the solutions, one asks further, whether the limit satisfies an equation similar to the ones one started out with. A main objective in homogenization theory is to show the convergence of the solutions as the ratio tends to zero and to derive the limit equation. 

If one assumes periodicity in the coefficients, many results are available for particular equations, see e.g.~\cite{BenLiPap,CioDon,TarIntro} as general references. In the non-periodic case one cannot expect a similar behavior as very simple equations show, see e.g.~\cite[p. x, equation $(*)$]{Waurick2011}. 

In this note, we discuss a general compactness result: Given a bounded sequence of coefficients, we prove that there exists a limit equation at least for a subsequence. We emphasize that we do not assume periodicity of the coefficients. Further, we show that this compactness result may be applied to coupled systems or to equations with memory terms. By means of an example, we elaborate on the precise strategy in Section \ref{section: disco}.

In the literature, there are many techniques available that allow the study of homogenization in the non-periodic case. We mention here the method of $H$-convergence in the sense of \cite{GuMo2,Murat1978,TarIntro} or \cite[Definition 13.3]{CioDon}, which is well-suited for elliptic equations. The method of $\Gamma$-convergence is tailored for variational integrals and optimization problems related to them, see e.g.~\cite{ABraid,B3,Mu1}. $G$-convergence and two-scale-convergence, see respectively \cite{Colombini1978,Spagnolo1968a,Gcon1} and \cite{Amar1998,Ng1,Wellander2009}, are concepts that may be applied to many equations of mathematical physics, as well.

The concept of $G$-convergence can be formulated in an abstract Hilbert space setting, see \cite[p.74]{Gcon1} (see also Section \ref{section: hom} in this paper). More precisely, given a sequence of continuously invertible operators $(A_n)_n$ its $G$-limit is an invertible operator $B$, if $(A_n^{-1})_n$ converges to $B^{-1}$ in the weak operator topology. In this general setting, it is unclear how to derive a more explicit expression for $B$. The notion of two-scale convergence uses a $L^2$-setting and does not apply to general Hilbert spaces. For a possible way of dealing with specific non-periodic coefficients, we refer to the generalization of two-scale convergence in \cite{Nguetseng2003,Nguetseng2004}.
 
 Here we provide an alternative way of discussing homogenization problems, which was introduced in \cite{Waurick2011,Waurick,WK2011} with (substantial) extensions in \cite{Waurick2012,W2013FP,W2014F,W2014G,W2013}. The idea bases on results of \cite{PicPhy,Picard}. In \cite{PicPhy,Picard}, a functional analytic Hilbert space framework is developed, which serves to derive a unified solution theory for many equations of mathematical physics. In this exposition, we will use the term \emph{evolutionary equations} to refer to the class described in \cite{PicPhy,Picard}. Note that evolutionary equations cover for example equations from thermoelasticity (\cite{MPTW2014}) or equations with dynamic boundary conditions (\cite{PSTW2015}) or equations typical in control theory with unbounded control and observation operators (\cite{PTW2014}).

Next, we sketch the functional analytic set up of evolutionary equations. For a given forcing term $f$, we consider
\begin{equation}\label{eq:int_1}
    \partial_t w+ A u = f.
\end{equation}
Here $\partial_t$ is a realization of the time-derivative as a normal \emph{continuously invertible} operator, $A$ is a skew-selfadjoint operator in some Hilbert space $H$ modeling the spatial derivatives. We want to solve \eqref{eq:int_1} for the unknown quantities $w$ and $u$. Clearly, the equation \eqref{eq:int_1} is under-determined. Thus, \eqref{eq:int_1} needs to be completed by a \emph{constitutive relation} or \emph{material law} $\s M$, being a continuous linear operator in time-space, which links  $w$ and $u$ via
\begin{equation}\label{eq:int_2}
   \s M u = w.
\end{equation}
Hence, we solve the equation 
\begin{equation}\label{eq:int_3}
   \partial_t \s M u + A u= f
\end{equation} for $u$.
In applications, the operator $\s M$ describes the material's properties, hence the name `material law': The operator $\s M$ consists of the inverse of the conductivity $\kappa$ if one discusses the heat equation. More precisely, the heat equation fits into \eqref{eq:int_3} with the settings
\[
   u=\begin{pmatrix}\theta \\ q\end{pmatrix}, \quad \s M = \begin{pmatrix} 1 & 0 \\ 0 & \partial_t^{-1}\kappa^{-1}\end{pmatrix}, \quad A=\begin{pmatrix} 0 & \diverg \\ \grad & 0\end{pmatrix},\quad f=\begin{pmatrix} Q \\ 0\end{pmatrix}
\]
where $\theta$ and $q$ are the heat and heat flux, respectively, and $Q$ is some external heat source. Indeed, with these setting, \eqref{eq:int_3} reads
\[
   \partial_t \begin{pmatrix} 1 & 0 \\ 0 & \partial_t^{-1}\kappa^{-1}\end{pmatrix} \begin{pmatrix} \theta \\ q \end{pmatrix}+ \begin{pmatrix} 0 & \diverg \\ \grad &  0 \end{pmatrix} \begin{pmatrix} \theta \\ q \end{pmatrix} = \begin{pmatrix} Q \\ 0 \end{pmatrix}.
\]
Reading off the first line, we obtain $\partial_t \theta +\diverg q = Q$. The second line is $q=-\kappa \grad \theta$. Hence, $\partial_t\theta -\diverg \kappa \grad \theta = Q$.

Maxwell's equations for the electro-magnetic field $(E,H)$ read
\[
 \left(\partial_t \begin{pmatrix} \eps & 0 \\ 0 &\mu \end{pmatrix} + \begin{pmatrix} \sigma & 0 \\ 0 & 0 \end{pmatrix}+ \begin{pmatrix} 0 & -\curl \\ \curl &  0 \end{pmatrix}\right) \begin{pmatrix} E\\ H \end{pmatrix} = \begin{pmatrix} J\\ 0 \end{pmatrix},
\]
where $\eps$ is the dielectricity, $\sigma$ is the (electric) conductivity, and $\mu$ is the magnetic permeability. $J$ are external currents. Thus, we obtain the shape of equation \eqref{eq:int_3} with the settings
\[
   u=\begin{pmatrix} E\\ H\end{pmatrix},\quad \s M =\begin{pmatrix} \eps & 0 \\ 0 & \mu \end{pmatrix} + \partial_t^{-1}\begin{pmatrix} \sigma & 0 \\ 0 & 0 \end{pmatrix}, \quad A = \begin{pmatrix} 0 & -\curl \\ \curl &  0 \end{pmatrix},\quad f= \begin{pmatrix} J\\ 0\end{pmatrix}.
\]
In general, we assume here that $\s M$ may be represented as a function of $\partial_t^{-1}$. Examples of such autonomous material laws are time-shifts, fractional time-derivatives or convolutions with respect to the temporal variable, see e.g.~\cite{KPSTW2014,Waurick,W2014F}. 

Our treatment of homogenization problems within this setting boils down to the discussion of continuous dependence on $\s M$ under a suitable topology, see also Section \ref{section: disco} for a more detailed discussion. Consider a sequence of material laws $(\s M_n)_n$ and corresponding solutions $(u_n)_n$ of the equation
\[
   \partial_t \s M_n u_n + Au_n = f.
\]
We ask, whether the sequence $(u_n)_n$ converges to some limit $v$ and whether there is a material law $\s N$, such that $v$ solves
\[
   \partial_t \s N v + A v = f.
\]
In \cite{Waurick2012}, this question was answered for $A$ with compact resolvent. It has been successfully applied to the heat equation, the wave equation or the visco-elastic equations with fractional time-derivatives or ordinary differential equations as constitutive relations, see \cite[Theorem 4.3 and Theorem 4.5]{Waurick2012} or \cite{W2014F}. In this paper, we complement the result obtained in \cite{Waurick2012,W2014F}. Imposing more restrictions on $\s M$, we merely require $A$ to have compact resolvent when restricted to a domain orthogonal to the nullspace of $A$. That is to say, instead of assuming that\footnote{For Hilbert spaces $H_1,H_2$ and a linear operator $A: D(A)\subseteqq H_1 \to H_2$ with domain $D(A)$, we denote the norm in the Hilbert space $H_1$ by $\abs{\cdot}_{H_1}$ and the graph norm of $A$ by $\abs{\cdot}_A$. If $H_1$ is continuously embedded in $H_2$, we write $H_1\hookrightarrow H_2$ or $(H_1,\abs{\cdot}_{H_1})\hookrightarrow (H_2,\abs{\cdot}_{H_2})$. If this embedding is compact we write $H_1\hookrightarrow\hookrightarrow H_2$ or $(H_1,\abs{\cdot}_{H_1})\hookrightarrow\hookrightarrow (H_2,\abs{\cdot}_{H_2})$.} $(D(A),\abs{\cdot}_A)\hookrightarrow \hookrightarrow (H,\abs{\cdot}_H)$, we only assume the following: Any sequence $(u_n)_n$ in $D(A)$ with both $(u_n)_n$ and $(Au_n)_n$ bounded in $H$ as well as $u_n\in N(A)^\bot$ consists of a $H$-norm convergent subsequence. We refer to the latter property as the \emph{nullspace-compactness-property} (or \emph{$(NC)$-property} for short) see also Corollary \ref{Co: Askew} below.

A relevant example for an operator satisfying the $(NC)$-property is the Maxwell operator $A=\begin{pmatrix} 0 & -\curl \\ {\interior{\curl}} & 0 \end{pmatrix}$ in $L_2(\Omega)^6$, where ${\interior{\curl}}$ is the $\curl$ operator with electric boundary condition and $\Omega\subseteqq \R^3$ is a bounded domain satisfying additional geometric requirements, cf.~\cite[Theorem 2]{Witsch1993} or \cite{Picard1984,Picard2001}; other boundary conditions may also be admitted, see \cite{BPS16}.

Furthermore, in \cite{Waurick2011} it is shown that the operator $\begin{pmatrix} 0& \diverg \\ {\interior{\grad}} & 0 \end{pmatrix}$ defined in $L_2(\Omega)^{N+1}$ satisfies the $(NC)$-property, where $\Omega\subseteqq \R^{N}$ is open and bounded, ${\interior{\grad}}$ is the distributional gradient in $L_2(\Omega)$ with domain equal to $W_{2,0}^1(\Omega)$. The superscript `` $\interior{}$ '' refers to Dirichlet boundary conditions and $\diverg$ is the negative adjoint of ${\interior{\grad}}$. The same reasoning can be applied to the spatial operator of the elastic equations $\begin{pmatrix} 0& \Grad^* \\ -\Grad & 0 \end{pmatrix}$, where $\Grad$ is the symmetrized gradient as defined in \cite[Definition 4.9]{Waurick2012} with some boundary conditions imposed on a bounded domain $\Omega$ satisfying suitable geometric requirements, cf.~e.g.~\cite[Theorem 2]{Weck1994}. 

Our main theorem,  Theorem \ref{Thm:(P2)}, may be seen as a general theorem giving a compactness result for the homogenization of (coupled) equations in mathematical physics. We shall also mention that the results obtained in this article not only generalize \cite[Theorem 2.3.14]{Waurick2011} but improve the representation of the homogenized equations. Moreover, we show that the homogenized equations satisfy the assumptions of Theorem \ref{Th: SolTh}, that is, we have a solution theory for the effective equations.

 A detailed discussion of our main result is given in Section \ref{section: disco}. In this section we also give an account of the ideas used and compare it to other strategies in the literature.

We sketch our plan to achieve our main result as follows. In Section \ref{section: set}, we discuss the mathematical framework of evolutionary equations and recall the main theorem of \cite{PicPhy}. Section \ref{section: top} sketches the ideas, definitions and main theorems of \cite{Waurick} and \cite{Waurick2012}. In Section \ref{section: hom} we present our main result, Theorem \ref{Thm:(P2)}. We show optimality of this theorem by means of counterexamples. The proofs in Section \ref{section: hom} also require the results from Section \ref{section: aux}, where some technical tools are provided. The results from Section \ref{section: aux} are needed to prove the well-posedness of the limit equation constructed in Theorem \ref{Thm:(P2)}. The abstract results obtained in Section \ref{section: hom} are exemplified in Section \ref{section: Ex}.
%

We indicate weak convergence in a Hilbert space by `$\rightharpoonup$' or `$\textnormal{w-}\lim$'. Norm-convergence will be denoted by `$\to$' if not specified differently. The underlying scalar field of any vector space discussed here is $\C$.

\section{Discussion of the main result}\label{section: disco}

%
%

Our aim is to provide homogenization results for a large class of partial differential equations. The homogenized coefficients, however, have a different representation as in classical results in the literature, cf.\ e.g.\ \cite{BenLiPap,Spagnolo1968a}. We illustrate the difference between the classical approach and the approach considered here with the heat equation (cf.\ e.g.\ \cite[p. 350]{Picard}):

Let $\Omega\subseteqq \R^N$ be a bounded domain. Denote by $\theta \colon (0,\infty)\times\Omega\to \R$ the heat distribution and by $q\colon (0,\infty)\times \Omega\to\R^N$ the heat flux. The heat equation is a system of the two equations
\begin{equation}\label{eq:phys}
  \begin{cases}
     \partial_t \theta + \diverg q = f & \\
     q = -\kappa \grad \theta. &
  \end{cases}
\end{equation}
Here $\partial_t \theta$ denotes the time-derivative of $\theta$, $f$ is a given source term and $\kappa \colon \R^N \to \R^{N\times N}$ is a bounded function being the (material dependent, symmetric) conductivity tensor satisfying $\kappa(x)\geqq c$ for some $c>0$ and all $x\in \R^N$.  Of course, $\theta$ and $q$ are the unknowns in the system. The first equation in \eqref{eq:phys} is called the \emph{heat flux balance} and the second one is \emph{Fourier's law}. The system is completed by boundary and initial conditions. For simplicity, we assume both homogeneous Dirichlet boundary conditions and homogeneous initial conditions for $\theta$. As a reminder for Dirichlet boundary conditions, we shall write ${\interior{\grad}}$ instead of $\grad$.

The classical way of discussing the heat equation is to substitute Fourier's law into the heat flux balance. Thus, the heat equation reads
\begin{equation}\label{eq:class}
  \partial_t \theta-\diverg\kappa{\interior{\grad}} \theta = f.
\end{equation}

Next, assume that $\kappa$ is $(0,1)^N$-periodic, that is, $\kappa(x+e)=\kappa(x)$ for all $x\in \mathbb{R}^N$, $e\in \mathbb{Z}^N$. In homogenization theory one is interested in the effective behavior of solutions of equations with highly oscillatory coefficients. A possible way to model that is the following. For $n\in\N$ consider the solutions $(\theta_n,q_n)$ and $\theta_n$ of the following respective equations
\begin{equation}\label{eq:phys_seq}
 \begin{cases}
     \partial_t \theta_n + \diverg q_n = f & \\
     q_n = -\kappa(n\cdot) {\interior{\grad}} \theta_n. &
  \end{cases}
\end{equation}
and
\begin{equation}\label{eq:sec_ord_seq}
  \partial_t{\theta}_n-\diverg\kappa(n\cdot){\interior{\grad}} \theta_n = f.
\end{equation}
Standard a priori estimates imply that (possibly after passing to a subsequence) $(\theta_n,q_n)_n$ and $(\theta_n)_n$ converge weakly to some functions $(\theta,q)$ and $\theta$, respectively. Classically, in order to determine the heat distribution $\theta$, one shows that $\theta$ solves $\partial_t{\theta}-\diverg\kappa_0{\interior{\grad}} \theta=f$. Here $\kappa_0$ is a well-known (constant-coefficient-)matrix. The main step in the classical approach is to prove 
\[\int_\Omega \kappa(n x){\interior{\grad}} \theta_n(t,x)\cdot {\interior{\grad}} \theta_n(t,x)\dd x \to \int_\Omega \kappa_0{\interior{\grad}}\theta(t,x) \cdot{\interior{\grad}} \theta (t,x)\dd x\quad (n\to\infty)\]
The tool for the proof of the latter is the famous ``$\diverg$-$\curl$-lemma'', which is due to Murat and Tartar, see e.g.\ \cite{Murat1978,Tartar1979}, or \cite{TarIntro,Briane2009}. The strategy of doing so is well-established and can also be applied to other cases such as linearized elasticity. 

Starting out with the sequence of equations given by \eqref{eq:phys_seq}, we shall sketch another way of deducing the limit equation. Written in a block-operator-matrix-form the equations \eqref{eq:phys_seq} read as
\begin{equation}\label{eq:phys_seq_syst}
  \left(\partial_t\begin{pmatrix}
             1 & 0 \\ 0 & 0
            \end{pmatrix}+\begin{pmatrix} 0 & 0 \\ 0 & \kappa(n\cdot)^{-1} \end{pmatrix} +\begin{pmatrix} 0 & \diverg  \\ {\interior{\grad}} & 0 \end{pmatrix}\right)\begin{pmatrix} \theta_n \\ q_n \end{pmatrix} = \begin{pmatrix} f \\ 0 \end{pmatrix}.
\end{equation}
Following the introduction and recalling that $\partial_t$ can be realized as a continuously invertible operator, the latter equations are clearly of the form \eqref{eq:int_3} with $\s M_n= \begin{pmatrix} 1 & 0 \\ 0 & 0  \end{pmatrix}+\partial_t^{-1}\begin{pmatrix} 0 & 0 \\ 0 & \kappa(n\cdot)^{-1} \end{pmatrix}$ and $A = \begin{pmatrix} 0& \diverg  \\ {\interior{\grad}} & 0 \end{pmatrix}$. Note that, due to the imposed homogeneous Dirichlet boundary conditions, $A$ is a skew-selfadjoint operator in $L_2(\Omega)^{N+1}$ since $\diverg^*=-{\interior{\grad}}$. 

For computing the limit as $n\to\infty$, we want to apply \cite[Theorem 3.5]{Waurick2011} (or Theorem \ref{Thm:(P1)} in this article). For this we require that $A$ has compact resolvent.

If $\Omega\subseteqq \R$ is a bounded, open interval, Theorem \ref{Thm:(P1)} is already applicable. Indeed, in this case the operator $A=\begin{pmatrix} 0& \partial_1  \\ {\interior{\partial}_1} & 0 \end{pmatrix}$ has compact resolvent. Thus, we infer that $\kappa(\cdot)^{-1}$ is periodic, and, hence, $\kappa(n\cdot)^{-1}$ converges in the weak*-topology of $L^\infty$ to the constant function $\int_0^11/\kappa$. The limit equation reads
\[
   \partial_t \theta - \partial_1 \Big(\int_0^11/\kappa\Big)^{-1} \interior{\partial}_1 \theta = f,
\]
where $\theta = \textnormal{w-}\lim_{n\to\infty} \theta_n$. Here, the homogenized coefficient, that is, $\Big(\int_0^11/\kappa\Big)^{-1}$ coincides with the \emph{harmonic mean of $\kappa$}.

Next, we consider the case $N\geqq 2$, that is, the underlying medium is at least $2$-dimensional. Then, the nullspace of $\diverg$ is the infinite-dimensional. Hence, the operator $\begin{pmatrix} 0& \diverg  \\ {\interior{\grad}} & 0 \end{pmatrix}$ has no compact resolvent. Thus, a rationale similar to the one-dimensional case fails.

In the following, we propose yet another reformulation of \eqref{eq:phys_seq} making Theorem \ref{Thm:(P1)} applicable. For this, note that the domain of ${\interior{\grad}}$ (endowed with its graph norm) equals $W_{2,0}^1(\Omega)$. By the selection theorem of Rellich and Kondrachov, we have $W_{2,0}^1(\Omega)\hookrightarrow\hookrightarrow L_2(\Omega)$, since $\Omega$ was assumed to be bounded. Next, we try to overcome the problem of the infinite-dimensional nullspace of $\diverg$ (and hence of $A$).  The idea is to restrict $A$ to a domain being orthogonal to the nullspace of $A$. Due to Dirichlet boundary conditions the operator ${\interior{\grad}}$ is one-to-one, thus $N(A)$, the nullspace of $A$, equals $\{0\}\oplus N(\diverg)\subseteqq L_2(\Omega)\oplus L_2(\Omega)^N$. Since $\diverg^*=-{\interior{\grad}}$, we have $L_2(\Omega)^N=\overline{R({\interior{\grad}})}\oplus N(\diverg)$. Using Poincar\'e's inequality, we deduce that the range of ${\interior{\grad}}$ is closed in $L_2(\Omega)^N$. Hence,
\[L_2(\Omega)^N = R({\interior{\grad}}) \oplus N(\diverg).\]
    Along that decomposition of $L_2(\Omega)^N$, we decompose the heat flux $q_n$ from equation \eqref{eq:phys_seq_syst} as $q_n=q_n^{(1)}+q_n^{(2)}$ with $q_n^{(1)}\in R({\interior{\grad}})$ and $q_n^{(2)}\in N(\diverg)$. We introduce the operator $P\colon L_2(\Omega)^N \to R({\interior{\grad}})$, which maps any $g\in L_2(\Omega)^N$ to its orthogonal projection $Pg$ in the range of ${\interior{\grad}}$. Then the adjoint of $P$ is the canonical embedding from $R({\interior{\grad}})$ into $L_2(\Omega)^N$. Hence, $A$ as an operator acting on $L_2(\Omega)\oplus R({\interior{\grad}})\oplus N(\diverg)$ may be written as follows
\[
   A = \begin{pmatrix} \tilde{A}  & 0  \\
                                              0    & 0 
\end{pmatrix}\text{ with } \tilde{A} = \begin{pmatrix}  0 & \diverg P^*  \\
                                              P \interior{\grad}    & 0 
\end{pmatrix}
\]
Observe that $A$ leaves the space $L_2(\Omega)\oplus R({\interior{\grad}})$ invariant. Moreover, note that $\tilde{A}$ is one-to-one and skew-selfadjoint on $L_2(\Omega)\oplus R({\interior{\grad}})$. Furthermore, it is not hard to see that the domain of $\diverg P^*$ if endowed with the graph norm is compactly embedded into $R({\interior{\grad}})$, see e.g.\ \cite[Lemma 4.1]{Waurick2012}. Thus, the operator $\tilde{A}$ has compact resolvent. 

Denoting by $Q\colon  L_2(\Omega)^N \to N(\diverg)$ the operator, which maps $g\in L_2(\Omega)^N$ to its orthogonal projection $Qg\in N(\diverg)$, we deduce from \eqref{eq:phys_seq_syst} the following system 
\begin{equation}\label{eq:phys_ext}
\left(\partial_t\begin{pmatrix}
             1 & 0 & 0 \\ 0 & 0 & 0 \\ 0&0&0
            \end{pmatrix}+\begin{pmatrix} 0 & 0  &  0 \\ 0 & P\kappa(n\cdot)^{-1}P^* & P\kappa(n\cdot)^{-1}Q^* \\
                                          0 & Q\kappa(n\cdot)^{-1}P^* & Q\kappa(n\cdot)^{-1}Q^* \end{pmatrix} +\begin{pmatrix} 0 & \diverg P^*  & 0 \\ P {\interior{\grad}}  & 0 & 0 \\
                                              0    & 0 & 0
\end{pmatrix}\right)\begin{pmatrix} \theta_n \\ q_n^{(1)} \\ q_n^{(2)} \end{pmatrix} = \begin{pmatrix} f \\  0 \\ 0 \end{pmatrix}. 
\end{equation}
Next, we could apply Theorem \ref{Thm:(P1)} to the first two rows of equation \eqref{eq:phys_ext} by putting $P\kappa(n\cdot)^{-1}Q^*q_n^{(2)}$ to the right-hand side, that is, 
\begin{equation*}
\left(\partial_t\begin{pmatrix}
             1 & 0  \\ 0 & 0  \end{pmatrix}+\begin{pmatrix} 0 & 0   \\ 0 & P\kappa(n\cdot)^{-1}P^*  \end{pmatrix} +\tilde{A} \right)\begin{pmatrix} \theta_n \\ q_n^{(1)}  \end{pmatrix} = \begin{pmatrix} f \\  -P\kappa(n\cdot)^{-1}Q^*q_n^{(2)} \end{pmatrix}. 
\end{equation*} If we let $n\to\infty$ in the latter formulation, we are not able to identify the limit of $-P\kappa(n\cdot)^{-1}Q^*q_n^{(2)}$.   A closer look into Theorem \ref{Thm:(P1)} reveals that due to the compactness of the resolvent of $\tilde{A}$ the sequence $(\theta_n,q_n^{(1)})_n$ converges in a way that the so-called `weak-strong principle' can be applied, cf.\ Theorem \ref{Le: Mat_law_ptwise}. Hence, we are led to express $q_n^{(2)}$ in terms of $(\theta_n,q_n^{(1)})$. Therefore, we perform similarity transformations of \eqref{eq:phys_ext}. We arrive at {\footnotesize
\begin{multline*}
   \left(\partial_t\begin{pmatrix}
             1 & 0 & 0 \\ 0 & 0 & 0\\ 0&0&0
            \end{pmatrix} +\begin{pmatrix} 0 & 0  &  0 \\ 0 & P\kappa(n\cdot)^{-1}P^*- P\kappa(n\cdot)^{-1}Q^*\left(Q\kappa(n\cdot)^{-1}Q^*\right)^{-1} Q\kappa(n\cdot)^{-1}P^* & 0 \\
                                          0 & Q\kappa(n\cdot)^{-1}P^* & Q\kappa(n\cdot)^{-1}Q^* \end{pmatrix}\right. \\
 \left.+\begin{pmatrix} 0 & \diverg P^*  & 0 \\ P {\interior{\grad}}  & 0 & 0 \\
                                              0    & 0 & 0
\end{pmatrix}\right)\begin{pmatrix} \theta_n \\  q^{(1)}_n \\  q^{(2)}_n \end{pmatrix} = \begin{pmatrix} f \\  0 \\ 0 \end{pmatrix}.
\end{multline*}}Multiplication by $\begin{pmatrix}
             1 & 0 & 0 \\ 0 & 1 & 0\\ 0&0&\left(Q\kappa(n\cdot)^{-1}Q^*\right)^{-1}
            \end{pmatrix}$ gives
\begin{multline*}
   \left(\partial_t\begin{pmatrix}
             1 & 0 & 0 \\ 0 & 0 & 0\\ 0&0&0
            \end{pmatrix}  +\begin{pmatrix} 0 & 0  &  0 \\ 0 & P\kappa(n\cdot)^{-1}P^*- P\kappa(n\cdot)^{-1}Q^*\left(Q\kappa(n\cdot)^{-1}Q^*\right)^{-1} Q\kappa(n\cdot)^{-1}P^* & 0 \\
                                          0 & \left(Q\kappa(n\cdot)^{-1}Q^*\right)^{-1} Q\kappa(n\cdot)^{-1}P^* & 1 \end{pmatrix}\right. \\
 \left.+\begin{pmatrix} 0 & \diverg P^*  & 0 \\ P {\interior{\grad}}  & 0 & 0 \\
                                              0    & 0 & 0
\end{pmatrix}\right)\begin{pmatrix} \theta_n \\  q^{(1)}_n \\  q^{(2)}_n \end{pmatrix} = \begin{pmatrix} f \\  0 \\ 0 \end{pmatrix}.
\end{multline*}
Note that the operators \[                                                                                                                                                                         \begin{pmatrix} 0 & 0  &  0 \\ 0 & P\kappa(n\cdot)^{-1}P^*- P\kappa(n\cdot)^{-1}Q^*\left(Q\kappa(n\cdot)^{-1}Q^*\right)^{-1} Q\kappa(n\cdot)^{-1}P^* & 0 \\
                                          0 & \left(Q\kappa(n\cdot)^{-1}Q^*\right)^{-1} Q\kappa(n\cdot)^{-1}P^* & 1 \end{pmatrix}\quad (n\in\N)                  \]
 form a bounded sequence in the space of linear operators in the separable Hilbert space $L_2(\Omega)^{N+1}$. Thus, there exists a subsequence, which converges in the weak operator topology of $L\left(L_2(\Omega)^{N+1}\right)$. Applying Theorem \ref{Thm:(P1)} to the first two rows of the latter equation, that is, 
   \begin{multline*}
\left(\partial_t\begin{pmatrix}
             1 & 0  \\ 0 & 0 
            \end{pmatrix}  +\begin{pmatrix} 0 & 0   \\ 0 & P\kappa(n\cdot)^{-1}P^*- P\kappa(n\cdot)^{-1}Q^*\left(Q\kappa(n\cdot)^{-1}Q^*\right)^{-1} Q\kappa(n\cdot)^{-1}P^* \end{pmatrix}\right. \\
 \left.+\begin{pmatrix} 0 & \diverg P^*  \\ P {\interior{\grad}}  & 0 \end{pmatrix}\right)\begin{pmatrix} \theta_n \\  q^{(1)}_n \end{pmatrix} = \begin{pmatrix} f \\  0 \end{pmatrix},
   \end{multline*}
we deduce that $(\theta_n,   q^{(1)}_n)_n$ weakly converges and that the sequence $(\partial_t^{-3}\theta_n(t),   \partial_t^{-3}q^{(1)}_n(t))_n$ \emph{strongly} converges  in $L_2(\Omega)\oplus R({\interior{\grad}})$ for all $t\in\R$, cf.\ Theorem \ref{Thm:(P1)}. Hence, \[\left(\left(Q\kappa(n\cdot)^{-1}Q^*\right)^{-1} Q\kappa(n\cdot)^{-1}P^*q_n^{(1)}\right)_n\] converges to the product of the limits of $\left(\left(Q\kappa(n\cdot)^{-1}Q^*\right)^{-1}Q\kappa(n\cdot)^{-1}P^*\right)_n$ and $(q_n^{(1)})_n$, see Corollary \ref{Cor:W-Strong-Product}. Thus, we may let $n\to\infty$ in
\[
   \left(Q\kappa(n\cdot)^{-1}Q^*\right)^{-1} Q\kappa(n\cdot)^{-1}P^*q_n^{(1)} + q_n^{(2)} = 0.
\]
We get that $(\theta_n,q^{(1)}_n,q^{(2)}_n)_n$ weakly converges to a solution of the following equation {\small 
\begin{multline*}
   \left(\partial_t\begin{pmatrix}
             1 & 0 & 0 \\ 0 & 0 & 0\\ 0&0&0
            \end{pmatrix} +\begin{pmatrix} 0 & 0  &  0 \\ 0 & \lim_{n\to\infty}\left(P\kappa(n\cdot)^{-1}P^*- P\kappa(n\cdot)^{-1}Q^*\left(Q\kappa(n\cdot)^{-1}Q^*\right)^{-1} Q\kappa(n\cdot)^{-1}P^*\right) & 0 \\
                                          0 & \lim_{n\to\infty}\left(\left(Q\kappa(n\cdot)^{-1}Q^*\right)^{-1} Q\kappa(n\cdot)^{-1}P^*\right) & 1 \end{pmatrix}\right. \\
 \left.+\begin{pmatrix} 0 & \diverg P^*  & 0 \\ P {\interior{\grad}}  & 0 & 0 \\
                                              0    & 0 & 0
\end{pmatrix}\right)\begin{pmatrix} \theta \\  q^{(1)} \\  q^{(2)} \end{pmatrix} = \begin{pmatrix} f \\  0 \\ 0 \end{pmatrix}.
\end{multline*}}

Next, we want to apply Theorem \ref{Th: SolTh} in order to obtain well-posedness of the limit equation. For this, we again apply similarity transformations. The details are given in Section \ref{section: aux}. We arrive at
\begin{multline}\label{eq:phys_limit}
   \left(\partial_t\left(\begin{smallmatrix}
             1 & 0 & 0 \\ 0 & 0 & 0\\ 0&0&0
            \end{smallmatrix}\right) \right.+ \\ \left. \left(\begin{smallmatrix} 0 & 0  &  0 \\ 0 & \lim_{n\to\infty}\left(P\kappa(n\cdot)^{-1}P^*- P\kappa(n\cdot)^{-1}Q^*\left(Q\kappa(n\cdot)^{-1}Q^*\right)^{-1} Q\kappa(n\cdot)^{-1}P^*\right) & \lim_{n\to\infty}\left( P\kappa(n\cdot)^{-1}Q^*\left(Q\kappa(n\cdot)^{-1}Q^*\right)^{-1}\right) \\
                                          0 & \lim_{n\to\infty}\left(\left(Q\kappa(n\cdot)^{-1}Q^*\right)^{-1} Q\kappa(n\cdot)^{-1}P^*\right) & \lim_{n\to\infty} \left(Q\kappa(n\cdot)^{-1}Q^*\right)^{-1} \end{smallmatrix}\right)\right. \\
 \left.+\left(\begin{smallmatrix} 0 & \diverg P^*  & 0 \\ P {\interior{\grad}}  & 0 & 0 \\
                                              0    & 0 & 0
\end{smallmatrix}\right)\right)\left(\begin{smallmatrix} \theta \\  q^{(1)} \\  q^{(2)} \end{smallmatrix}\right) = \left(\begin{smallmatrix} f \\  0 \\ 0 \end{smallmatrix}\right).
\end{multline}
It can be shown that Theorem \ref{Th: SolTh} applies to the latter equation yielding the limit equation to be well-posed. Note that, if one is only interested in the behavior of the heat distribution $\theta$, we can reformulate the latter equation into a second order form. The resulting equation would be
{\small
\[
  \partial_t \theta - \diverg P^*\left(\lim_{n\to\infty}P\kappa(n\cdot)^{-1}P^*- P\kappa(n\cdot)^{-1}Q^*\left(Q\kappa(n\cdot)^{-1}Q^*\right)^{-1} Q\kappa(n\cdot)^{-1}P^*\right)^{-1}P {\interior{\grad}} \theta = f.
\]}Using the periodicity of $\kappa$, we deduce that $P\kappa(n\cdot)^{-1}P^*$ converges in the weak operator topology to $P \int_{(0,1)^N} \kappa(x)^{-1}\dd x P^*$, cf. \cite[Proposition 4.3]{Waurick2012}. 
Hence, the limit equation reads

{\footnotesize
\[
  \partial_t \theta - \diverg P^*\left(P\int_{(0,1)^N}\kappa^{-1}P^*- \lim_{n\to\infty}P\kappa(n\cdot)^{-1}Q^*\left(Q\kappa(n\cdot)^{-1}Q^*\right)^{-1} Q\kappa(n\cdot)^{-1}P^*\right)^{-1}P {\interior{\grad}} \theta = f.
\]}

The interested reader might think, why such a seemingly complicated strategy yielding the homogenized equations should be applied. In the case of the heat equation this strategy indeed does not give anything new despite the fact that the homogenized equations have a different representation. Further, with this strategy one cannot easily deduce the convergence of the whole sequence. However, note that the approach presented here only uses abstract theory from functional analysis and does not rely on the specific form of $\kappa$ being a periodic multiplication operator. If $\kappa$ is a linear operator invoking non-local terms, well-known homogenization theory might fail to work. Moreover, the way of computing the homogenized coefficients carries over to a large class of evolutionary equations: It is possible to treat Maxwell's equations, the wave equation, the heat equation or general coupled systems in mathematical physics in a \emph{unified} manner. It is also possible that a second order formulation might not be available or is not easy to handle, cf.\ e.g.\ \cite[Equation (6.3.9), p. 455]{Picard} so that the usual strategy might not work. 

\section{Setting}\label{section: set}

We recall the setting of evolutionary equations established in \cite{PicPhy} or \cite[Chapter 6]{Picard}. For the construction of the time-derivative $\partial_t$, we particularly refer to \cite{KPSTW2014}. Let $H$ be a Hilbert space and denote by $L_2(\R;H)$ the space of $H$-valued $L_2$-functions. The operator
\[
  \partial : W_2^1(\R;H)\subseteqq L_2(\R;H)\to L_2(\R;H): f\mapsto f'
\]assigning to each weakly differentiable $H$-valued function its weak derivative is skew-selfadjoint. Define the unitary Fourier transform $\s F: L_2(\R;H)\to L_2(\R;H)$ as the closure of the mapping
\[
   f\mapsto\left(\xi\mapsto  \frac{1}{\sqrt{2\pi}}\int_{\R} e^{-ix\xi}f(x)\dd x\right)
\] defined for $f\in C_c^\infty(\R;H)$. Let 
\begin{multline*}
  m: \{f\in L_2(\R;H); (x\mapsto xf(x))\in L_2(\R;H)\}\subseteqq \\ L_2(\R;H)\to L_2(\R;H): f\mapsto (x\mapsto xf(x)).
\end{multline*}
For $\nu>0$ define $H_{\nu,0}(\R;H)\coloneqq L_2(\R,\exp(-2\nu x)\dd x;H)$ the space of $H$-valued (equivalence classes of) square-integrable functions with respect to the weighted Lebesgue measure with Radon-Nikodym derivative $\exp(-2\nu (\cdot))$. We also write $H_{\nu,0}(\R)$ if $H=\C$. The mapping $\exp(-\nu m): H_{\nu,0}(\R;H)\to L_2(\R;H):f\mapsto (x\mapsto \exp(-\nu x)f(x))$ is unitary and the operator
\[
   \partial_{t,\nu}\coloneqq \exp(-\nu m)^*(\partial+ \nu) \exp(-\nu m) 
\]
is normal in $H_{\nu,0}(\R;H)$. If there is no risk of confusion, we simply write $\partial_t$ instead of $\partial_{t,\nu}$. We have $\partial_{t,\nu}^{-1} \in L(H_{\nu,0}(\R;H))$ with $\Abs{\partial_{t,\nu}^{-1}}=1/\nu$. Introducing the \emph{Fourier-Laplace transform} $\s L_\nu\coloneqq \s F \exp(-\nu m)$, we get
\[
   \partial_{t,\nu} = \s L_\nu^* (im+\nu)\s L_\nu.
\]
Consequently,
\[
    \partial_{t,\nu}^{-1} = \s L_\nu^* \left({im+\nu}\right)^{-1}\s L_\nu.
\]
The latter equation gives a functional calculus for the normal operator $\partial_{t,\nu}^{-1}$: 

\begin{Def}[Hardy space and functional calculus for $\partial_{t,\nu}$] For an open set $E\subseteqq \C$ and a Banach space $X$, we define the Hardy space 
\[
    \s H^\infty(E; X)\coloneqq \{ M \colon E\to X; M\text{ bounded, analytic}\}
\]and $\Abs{M}_\infty\coloneqq\sup\{\abs{M(z)}_X; z\in E\}$. Let $H_1,H_2$ be Hilbert spaces, $\nu>0$, $r>1/(2\nu)$.
\begin{itemize}
\item For $M\in \s H^\infty(B_{\C}(r,r);L(H_1,H_2))$ define
\[
   M(\partial_{t,\nu}^{-1})\coloneqq\s L_\nu^*\left(M \left(\frac1{im+\nu}\right)\right)\s L_\nu,
\]
where $\left(M\left(\frac1{im+\nu}\right)\phi\right)(t)=\left(M\left(\frac1{it+\nu}\right)\right)\left(\phi(t)\right)$ for all $t\in\R$ and $\phi \in C_c^\infty(\R;H_1)$.
\item For $c>0$ define 
\begin{multline*}
  \s H^{\infty,c}(B_\C(r,r);L(H_1))\coloneqq\\
  \{ M\in \s H^\infty(B_\C(r,r);L(H_1)); \Re z^{-1}M(z)\geqq c \; (z\in B_\C(r,r))\}.
\end{multline*}
\end{itemize}
For easy reference, we call elements of $\s H^\infty(B_\C(r,r);L(H_1))$ \emph{material laws} or \emph{constitutive relations} and elements of $\s H^{\infty,c}(B_\C(r,r);L(H_1))$ \emph{$(c)$-material laws}.
\end{Def}

We will deal with operators in $H_{\nu,0}(\R;H)$ in the following. In consequence, we identify any closed, densely defined operator $A:D(A)\subseteqq H\to H$ in some Hilbert space $H$ with its canonical extension on the space of $H$-valued $H_{\nu,0}(\R)$-functions, cf. \cite{PicPhy}. We have the following well-posedness theorem taken from \cite{PicPhy}.
\begin{Sa}[{\cite[Solution theory]{PicPhy}}]\label{Th: SolTh} Let $H$ be a Hilbert space, $c,\nu>0, r>1/(2\nu)$ and $M\in \s H^{\infty,c}(B(r,r);L(H))$. Let $A: D(A)\subseteqq H\to H$ be skew-selfadjoint. Then the equation
\[
   (\partial_{t}M(\partial_{t}^{-1})+A)u=f
\]
admits a unique solution $u\in H_{\nu,0}(\R;H)$ for all $f\in D$ for some $D\subseteqq H_{\nu,0}(\R;H)$ dense. Moreover, the solution operator $(\partial_t M(\partial_t^{-1})+A)^{-1}$ is a densely defined, continuous operator in $H_{\nu,0}(\R;H)$ with operator norm bounded by $\frac 1c$, and the operator $\overline{(\partial_t M(\partial_t^{-1})+A)^{-1}}$ is causal, that is, for all $f\in H_{\nu,0}(\R;H)$ and $a\in \R$ we have
\[
  \chi_{(-\infty,a)}\overline{(\partial_t M(\partial_t^{-1})+A)^{-1}}(\chi_{(-\infty,a)}f) = \chi_{(-\infty,a)}\overline{(\partial_t M(\partial_t^{-1})+A)^{-1}}(f),
\]
where $\chi_{(-\infty,a)}$ denotes the multiplication operator mapping $f\in H_{\nu,0}(\R;H)$ to the truncated function $t\mapsto \chi_{(-\infty,a)}(t)f(t)$.
\end{Sa}

\begin{rem}
  The latter theorem may be generalized to non-autonomous equations, see \cite{W2015,PTWW2013}.
\end{rem}

We note that the results in Theorem \ref{Th: SolTh} results carry over to ``tailor made'' distribution spaces -- so-called Sobolev lattices -- discussed in \cite{PicSoLa}. In \cite[Remarks 1.2: (i)--(iii)]{Waurick2012} and \cite[Sections 2 and 3]{PicPhy} the core issues are sketched. We will use the notation from \cite{Waurick2012} and for the sake of clarity, we recall the main definitions. For $k\in \Z$, a Hilbert space $H$ and a densely defined, closed linear operator $C:D(C)\subseteqq H\to H$ with $0 \in \rho(C)$, we denote by $H_k(C)$ the Hilbert space defined as the completion of $D(C^{\abs{k}})$ with respect to the norm $\abs{\cdot}_{H_k(C)}:u\mapsto \abs{C^ku}_H$. It can be shown that the closure of $H_{\abs{k}}(C)\subseteqq H_k(C)\to H_{k-1}(C): u\mapsto Cu$ is unitary. We will re-use the letter $C$ for this extension. We are interested in the special cases $C=A+1$ with $A$ skew-selfadjoint or $C=\partial_t$. For $\ell\in\{-1,0,1\}$ we let $H_{\ell,A}\coloneqq H_\ell(A+1)$. For $\partial_t$ defined on $H_{\nu,0}(\R)$-functions with values in a Hilbert space $H$ we write $H_{\nu,k}(\R;H)\coloneqq  H_{k}(\partial_t)$. Consequently, we also use the spaces $H_{\nu,k}(\R;H_{\ell,A})$, $\ell\in\{-1,0,1\}$. The extension of the solution operator to $H_{\nu,-1}(\R;H)$ also serves as a way to model initial value problems, see e.g.~\cite[Section 6.2.5]{Picard}.

\section{Preliminary results}\label{section: top}

We summarize some findings from \cite{Waurick,Waurick2012}.
\begin{Def}
 For an open set $E\subseteqq \C$ and Hilbert spaces $H_1,H_2$, we define on the set $\s H^\infty(E;L(H_1,H_2))$ the initial topology $\tau_{\s M}$ induced by the mappings
\[
   \s H^\infty(E;L(H_1,H_2)) \ni M \mapsto \langle \phi, M(\cdot)\psi\rangle \in \s H(E),
\]
where $\s H(E)$ is the set of holomorphic $\C$-valued functions endowed with the compact open topology. We define $\s H_{\textnormal{w}}^\infty(E;L(H_1,H_2))\coloneqq \left(\s H^\infty(E;L(H_1,H_2)),\tau_{\s M}\right)$ and re-use the name $\s H_{\textnormal{w}}^\infty(E;L(H_1,H_2))$ for the underlying set. 
\end{Def}
\begin{Sa}[sequential compactness, {\cite[Theorem 3.4]{Waurick}}]\label{Thm: seq_comp} Let $H_1,H_2$ be separable Hilbert spaces, $E \subseteqq \C$ open. Let $\s B\subseteqq \s H_{\textnormal{w}}^{\infty}(E;L(H_1,H_2))$ be \emph{bounded}, that is, \[\sup\{\Abs{M(z)}_{L(H_1,H_2)};z\in E, M\in \s B\}<\infty.\] Then $\s B$ is relatively sequentially compact.
\end{Sa}
\begin{Le}[{\cite[Lemma 3.5]{Waurick}}] \label{Le: loc_uni_weak} Let $H$ be a Hilbert space, $r>0$.
Let $(M_n)_n$ be a bounded and convergent sequence in the space $\s H_{\textnormal{w}}^{\infty}(B(r,r);L(H_1,H_2))$ with limit $M
\in \s H_{\textnormal{w}}^{\infty}(B(r,r);L(H_1,H_2))$. Then
$(M_n(\partial_t^{-1}))_n$ converges to $M(\partial_t^{-1})$ in the
weak operator topo\-logy of $L(H_{\nu,k}(\R;H_1),H_{\nu,k}(\R;H_2))$, where $\nu>
1/(2r)$, $k\in \Z$.
\end{Le}
\begin{proof}
 In {\cite[Lemma 3.5]{Waurick}}, the claim was shown for the case $k=0$ and $H_1=H_2$. The general case follows by observing that $\partial_t^{k}: H_{\nu,k}(\R;H_1)\to H_{\nu,0}(\R;H_1)$ is unitary and obvious modifications.
\end{proof}
\begin{Le}[{\cite[Lemma 1.5]{Waurick2011}}]\label{Le: Conv_of_coeff} Let $H_1,H_2$ be Hilbert spaces. Let $E\subseteqq \C$ be an open disc with center $z$ and let $(M_n)_n=(\sum_{k=0}^\infty (\cdot-z)^kA_{nk})_n$ be a convergent sequence in $\s H_{\textnormal{w}}^{\infty}(E;L(H_1,H_2))$ with limit $\sum_{k=0}^\infty (\cdot - z)^kA_k$. Then $A_{nk}\to A_k$ as $n\to\infty$ in the weak operator topology of $L(H_1,H_2)$ for all $k\in\N_0$.
\end{Le}
For a Hilbert space $H$ and $\nu>0$, we define
\[
   C_\nu(\R;H)\coloneqq \{\phi\in C(\R;H); \sup_{t\in\R}\abs{\exp(-\nu t)\phi (t)}_H<\infty\}.
\]
We endow $C_\nu(\R;H)$ with the norm $\abs{\cdot}_{C_\nu}:\phi\mapsto \sup_{t\in\R}\abs{\exp(-\nu t)\phi (t)}_H$. Recall from \cite[Lemma 3.1.59]{Picard} that $H_{\nu,1}(\R;H)$ continuously embeds into $C_\nu(\R;H)$. 
\begin{Le}[{\cite[Lemma 2.2]{Waurick2012}}]\label{Le: part_ptw} Let $H$ be a Hilbert space, $\nu>0$. If $(f_n)_n$ in $H_{\nu,1}(\R; H)$ is bounded and
converges pointwise to some $f\in H_{\nu,1}(\R;H)$, then
\[
   \partial_t^{-1}f_n(t) \tor n\infty \partial_t^{-1} f(t),
\]
for all $t\in \R$.
\end{Le}

\begin{Sa}[weak-strong principle, {\cite[Theorem 2.3]{Waurick2012}}]\label{Le: Mat_law_ptwise} Let $H$ be a Hilbert space, $\eps>0$, $(M_n)_n$ be a convergent sequence in $\s H^\infty_{{\textnormal{w}}}(B_\C(0,\eps);L(H_1,H_2))$ with limit $M$. Then, for $\nu>2/\eps$ and any bounded sequence $(v_n)_{n}$ in
$H_{\nu,1}(\R;H_1)$ and $v\in H_{\nu,1}(\R;H_1)$ such that
$v_n(t)\tor n\infty v(t)$ in $H_1$ for all $t\in\R$,
\[
    \textnormal{w-}\lim_{n\to\infty}(M_n(\partial_t^{-1})v_n)(t) = (M(\partial_t^{-1})v)(t)\in H_2,
\]
for all $t\in\R$.
\end{Sa}
\begin{proof}
 In \cite{Waurick2012} the proof is given for the case $H_1=H_2$. The assertion follows analogously with obvious modifications.
\end{proof}

\begin{Fo}\label{Cor:W-Strong-Product} Let $H_1,H_2$ be Hilbert spaces, $\eps>0$, $(M_n)_n$ be a convergent sequence in $\s H^\infty_{{\textnormal{w}}}(B_\C(0,\eps);L(H_1,H_2))$ with limit $M\in\s H^\infty_{\textnormal{w}}(B_\C(0,\eps);L(H_1,H_2))$. Let $\nu>2/\eps$, $k\in\Z$ and let $(v_n)_{n}$ be bounded in $H_{\nu,k}(\R;H_1)$, $v\in H_{\nu,k}(\R;H_1)$. Assume there is $l\in \N_0$ such that $\partial_t^{-l}v_n\in H_{\nu,1}(\R;H_1)$ and
$\partial_t^{-l}v_n(t)\tor n\infty \partial_t^{-l}v(t)$ in $H_1$ for all $t\in\R$. Then
\[
    \textnormal{w-}\lim_{n\to\infty}M_n(\partial_t^{-1})v_n = M(\partial_t^{-1})v\in H_{\nu,k}(\R;H_2).
\] 
\end{Fo}
\begin{proof}
 Since $(M_n(\partial_t^{-1})v_n)_n$ is bounded in $H_{\nu,k}(\R;H_2)$, there is a subsequence with indices $(n_j)_j$ weakly converging to some $w\in H_{\nu,k}(\R;H_2)$. The assumption guarantees that $(\partial_t^{-\abs{k}-l}v_n)_n$ is bounded in $H_{\nu,1}(\R;H_1)$. Moreover, by Lemma \ref{Le: part_ptw}, $(\partial_t^{-\abs{k}-l}v_n)_n$ converges pointwise to $\partial_t^{-\abs{k}-l}v$. Thus, by Theorem \ref{Le: Mat_law_ptwise} and the weak continuity of point-evaluation, we deduce that, for $t\in\R$,
\begin{align*}
   (\partial_t^{-\abs{k}-l}w)(t) &= \textnormal{w-}\lim_{j\to\infty}(\partial_t^{-\abs{k}-l}M_{n_j}(\partial_t^{-1})v_{n_j})(t) \\
                                 &= \textnormal{w-}\lim_{j\to\infty}M_{n_j}(\partial_t^{-1})\partial_t^{-\abs{k}-l}v_{n_j}(t) \\
                                 &= M(\partial_t^{-1})\partial_t^{-\abs{k}-l}v(t) =\partial_t^{-\abs{k}-l} M(\partial_t^{-1})v(t).
\end{align*}
Hence, $w=M(\partial_t^{-1})v$.
\end{proof}
\section{A general compactness theorem for the homogenization of evolutionary equations}\label{section: hom}

%
We introduce the concept of $G$-convergence to bridge the gap between the classical approach to homogenization theory and the Hilbert space perspective discussed here. 

\begin{Def}[$G$-convergence, {\cite[p.\ 74]{Gcon1}}] Let $H$ be a Hilbert space. Let $(A_n:D(A_n)\subseteqq H\to H)_n$ be a 
sequence of one-to-one mappings onto $H$ and let $B:D(B) \subseteqq H\to H$ be one-to-one.
We say that $(A_n)_n$ \emph{$G$-converges to $B$} if for all $f\in H$ the sequence $(A_n^{-1}(f))_n$ 
converges weakly to some $u$, which satisfies $u\in D(B)$ and $B(u)=f$. $B$ is called a \emph{$G$-limit} 
of $(A_n)_n$. We say that $(A_n)_n$ \emph{strongly $G$-converges to $B$} in $H$, if for all weakly 
converging sequences $(f_n)_n$ in $H$, $(A_n^{-1}(f_n))_n$ weakly converges to some $u$, which satisfies $u\in D(B)$ 
and $B(u)=\textnormal{w-}\lim_{n\to\infty}f_n$.
\end{Def}
\begin{Prop}\label{Prop:G_unique}  The $G$-limit is uniquely determined. 
\end{Prop}
\begin{proof} Let $H$ be a Hilbert space. Let $(A_n)_n$ be a sequence of one-to-one onto mappings which is $G$-convergent to the one-to-one mapping $B:D(B)\subseteqq H\to H$.
 Define $C\coloneqq  \{ (u,f)\in H\oplus H; u=\textnormal{w-}\lim_{n\to\infty} A_n^{-1}(f)\}$. Then $C\subseteqq B$, so that $C$ is a mapping. 
Moreover, since $C$ is onto and $B$ is one-to-one, we conclude that $C=B$. 
\end{proof}

\begin{rems}\label{rem:wot} Assume, in addition, that $(A_n)_n$ in the above definition is a sequence of linear and closed operators. Further assume $B$ to be closed and linear. Then the above definition of $G$-convergence is precisely convergence of the resolvents in the weak operator topology, which is the original definition in \cite{Gcon1} in the Hilbert space setting.  
\end{rems}

We now prove compactness results concerning $G$-convergence for operators that are associated with
evolutionary equations. More precisely, we will deal with the following cases:

\begin{Def} Let $H_1,H_2$ be Hilbert spaces. We say a pair $((M_n)_n,\s A)$ satisfies 
\begin{enumerate}
 \item[\textbf{(P1)}]\label{p1} if there exists $\eps,r,c >0$ such that $(M_n)_n$ is a bounded sequence in \[\s H^\infty(B(0,\eps);L(H_1))\cap \s H^{\infty,c}(B(r,r);L(H_1))\] and \[\s A:D(\s A)\subseteqq H_1\to H_1\] is skew-selfadjoint and the embedding $(D(\s A),\abs{\cdot}_{\s A})\hookrightarrow (H_1,\abs{\cdot}_{H_1})$ is compact,
 \item[\textbf{(P2)}]\label{p4} if there exists $\eps,c,r>0$ such that $(M_n)_n = \left( \begin{pmatrix} M_{11,n} & M_{12,n} \\ M_{21,n} & M_{22,n} \end{pmatrix}\right)_n$ is bounded in $\s H^\infty(B(0,\eps);L(H_1\oplus H_2))\cap \s H^{\infty,c}(B(r,r);L(H_1\oplus H_2))$ and $\s A = \begin{pmatrix} A & 0 \\ 0 & 0 \end{pmatrix}$ is such that $((M_{11,n})_n,A)$ satisfies (P1). Moreover,
\ben
   \item for all $n\in\N$, $R(M_1(0))=R(M_n(0))$ and $M_n(0)\geqq c$ on $R(M_1(0))$,
   \item denoting by $q_j:H_j \to R(\pi_j^*)\cap N(M_1(0))$ $(j\in\{1,2\})$ the canonical ortho-projections, we have for all $n\in\N$ \[\left(\left(q_2M_{22,n}'(0)q_2^*\right)^{-1}q_2M_{21,n}'(0)q_1^{*}\right)^*=q_1M_{12,n}'(0)q_2^{*}\left(q_2M_{22,n}'(0)q_2^*\right)^{-1}.\] 
\een
\end{enumerate}
\end{Def}

With these definitions, the core result in \cite{Waurick2012} now reads as follows.
\begin{Sa}[{\cite[Theorem 3.5]{Waurick2012}}]\label{Thm:(P1)} Let $H$ be a Hilbert space and assume that $((M_n)_n,\s A)$ satisfies (P1) and that $(M_n)_n$ converges to $N\in \s H_{\textnormal{w}}^\infty(B(0,\eps);L(H))$.  Then there exists $\nu_0\geqq 0$ such that for all $\nu>\nu_0$, $(\partial_tM_{n_k}(\partial_t^{-1})+\s A)_k$ strongly $G$-converges to $\partial_t N(\partial_t^{-1})+\s A$ in $H_{\nu,-1}(\R;H)$. Moreover, $N\in \s H^{\infty,c}(B(r,r);L(H))$ and
\[
   \partial_t^{-3}(\partial_t M_{n}(\partial_t^{-1})+\s A)^{-1}f_n(t) \to \partial_t^{-3}(\partial_t N(\partial_t^{-1})+\s A)^{-1}(\textnormal{w-}\lim_{n\to\infty}f_n)(t)\in H
\]
  as $n\to\infty$ for all $t\in\R$ and all weakly convergent sequences $(f_n)_n$ in $H_{\nu,-1}(\R;H)$. 
\end{Sa}

The generalization of this theorem to the case $(P2)$ requires a homogenization result for the case of $A=0$, that is to say, a result on the homogenization of ordinary integro-differential equations. Since we deal with a possibly degenerate case in the sense of \cite[Section 3.3]{PicPhy}, we cannot use the homogenization result for ordinary integro-differential equations already established in \cite[Theorem 5.2]{Waurick}. The refined argument is tailored for the $0$-analytic case (cf.\ Section \ref{section: aux}), which, however, does not cover the results in \cite{Waurick}, see also \cite{W2014G} for a related result.

\begin{Sa}\label{Th: 0-anal-ode} Let $H$ be a separable Hilbert space, $\eps,c,d,r>0$. Let $(M_n)_n$ be a bounded sequence in $\s H^\infty(B(0,\eps);L(H))\cap \s H^{\infty,c}(B(r,r);L(H))$ and assume that for all $n\in \N$, $M_n(0)\geqq d$ on $R(M_n(0))=R(M_1(0))$. Then there exists $r'\in(0,r]$ and a strictly monotone sequence of positive integers $(n_k)_k$ such that, for $\nu>1/(2r')$, $(\partial_t M_{n_k}(\partial_t^{-1}))_k$ $G$-converges to $\partial_t\mu(\partial_t^{-1})$ in $H_{\nu,-1}(\R;H)$, where $\mu$ has the following properties: there is $\eps',c'>0$ such that
\ben
 \item $\mu \in \s H^{\infty}(B(0,\eps');L(H))\cap \s H^{\infty,c'}(B(r',r');L(H))$,
 \item $R(\mu(0))=R(M_1(0))$,
 \item for all open $E\subseteqq \C$ relatively compact in $B(0,\eps')\setminus \{0\}$ ($:\iff E\subset\subset B(0,\eps')\setminus \{0\}$),  
\[
  M_{n_k}(\cdot)^{-1} \to \mu(\cdot)^{-1} \in \s H^{\infty}_\textnormal{w}(E;L(H)) \quad (k\to\infty).
\]
\een
\end{Sa}
\begin{proof}
 Define the Hilbert spaces $H_1\coloneqq R(M_1(0))$ and $H_2\coloneqq N(M_1(0))$ together with the canonical (orthogonal) projections $\pi_j:H\to H_j$, $j\in\{1,2\}$. Then, for all $n\in\N$ and $j,k\in\{1,2\}$, set $M_{jk,n}(\cdot)\coloneqq  \pi_j M_n(\cdot)\pi_k^*$. Now, the first assertion in Lemma \ref{Le: prototype} ensures the existence of $\eps'>0$ such that, for all $E\subseteqq \C$ relatively compact in $B(0,\eps')\setminus \{0\}$, the sequence $(M_n(\cdot)^{-1})_n$ is bounded in $\s H^{\infty}(E;L(H_1\oplus H_2))$. By $\sigma$-compactness of $B(0,\eps')\setminus \{0\}$ and Theorem \ref{Thm: seq_comp}, we may choose a subsequence $(M_{n_k}(\cdot)^{-1})_k$ of $(M_{n}(\cdot)^{-1})_n$ such that there is a holomorphic mapping $\eta:B(0,\eps')\setminus\{0\}\to L(H)$ with
\[
  M_{n_k}(\cdot)^{-1} \to \eta \in \s H_{\textnormal{w}}^{\infty}(E;L(H)) \quad ( k\to\infty, E\subset\subset B(0,\eps')\setminus\{0\}).
\]
 By Cauchy's integral formulas, we infer that the coefficients of the Laurent series expansions of $M_{n_k}(\cdot)^{-1}$ converge in the weak operator topology $\tau_{\textnormal{w}}$ to the respective ones of $\eta$. Hence, with the help of the first assertion of Lemma \ref{Le: prototype}, the Laurent series expansion of $\eta$ is of the form
\[
   \eta (z) = \begin{pmatrix}
                (\tau_{\textnormal{w}}\text{-})\lim_{k\to\infty} M_{11,n_k}(0)^{-1}+\hat M_{11}(z) & \hat M_{12}(z)\\
                \hat M_{21}(z) & z^{-1} (\tau_{\textnormal{w}}\text{-})\lim_{k\to\infty} M_{22,n_k}(0)^{-1} + \hat M_{22}(z) 
              \end{pmatrix}
\]
for suitable bounded holomorphic operator-valued functions $\hat M_{jk}$ for $j,k\in\{1,2\}$. The second assertion of Lemma \ref{Le: prototype} yields the existence of $\eps''>0$ such that $\mu\coloneqq \eta(\cdot)^{-1} \in \s H^{\infty}(B(0,\eps'');L(H))$. Moreover, from the representation in Lemma \ref{Le: prototype}, we read off that $R(M_1(0))=R(\mu(0))$ and $\mu(0)\geqq d'$ on $H_1$ for some $d'>0$ according to Inequality \eqref{eq:alm_alm_triv} and the fact that positive definiteness is preserved under limits in the weak operator topology. Similarly, $\Re \mu'(0) \geqq c'>0$ on $H_2$. Thus, by Remark \ref{rem:conv} it follows that $\mu$ lies in $\s H^{\infty,c''}(B(r',r');L(H))$ for some $r',c''>0$. It remains to show the $G$-convergence result. To this end let $\nu>1/(2r')$. By the convergence of the coefficients in the Laurent series of $((M_{n_k}(\cdot))^{-1})_k$, we get that $((\cdot)(M_{n_k}(\cdot))^{-1})_k$ converges to $(\cdot)\eta(\cdot)$ in $\s H^\infty_{\textnormal{w}}(B(1/(2\nu),1/(2\nu));L(H))$. Thus, Lemma \ref{Le: loc_uni_weak} implies that $((\partial_tM_{n_k}(\partial_t^{-1}))^{-1})_k$ converges to $\partial_t^{-1}\eta(\partial_t^{-1})$ in the weak operator topology of $L(H_{\nu,-1}(\R;H))$. Employing Remark \ref{rem:wot}, we obtain the desired $G$-convergence. 
\end{proof}

\begin{Sa}\label{Thm:(P2)} Let $H_1,H_2$ be separable Hilbert spaces. Assume that $((M_n)_n,\s A)$ satisfies (P2). Then there exists $\nu_0\geqq 0,\eps',c'>0$ and $(n_k)_k$ a strictly monotone sequence of positive integers
 such that for all $\nu>\nu_0$ the sequence $(\partial_tM_{n_k}(\partial_t^{-1})+\s A)_k$ $G$-converges to $(\partial_t N(\partial_t^{-1})+\s A)$ in $H_{\nu,-1}(\R;H_1\oplus H_2)$ with
\begin{multline*}N(\cdot)\coloneqq \begin{pmatrix} \eta_{1}(\cdot) + \eta_{4}(\cdot)\eta_2(\cdot)^{-1}\eta_{3}(\cdot) & \eta_{4}(\cdot)\eta_2(\cdot)^{-1} 
\\ \eta_2(\cdot)^{-1}\eta_{3}(\cdot) & \eta_2(\cdot)^{-1} \end{pmatrix} \\
\in \s H^{\infty}(B(0,\eps');L(H_1\oplus H_2))\cap \s H^{\infty,c'}(B(1/(2\nu_0),1/(2\nu_0));L(H_1\oplus H_2)),
\end{multline*} where
\begin{align*}
   \eta_1(\cdot) &\coloneqq  \lim_{k\to\infty}M_{11,n_k}(\cdot)- M_{12,n_k}(\cdot)M_{22,n_k}(\cdot)^{-1}M_{21,n_k}(\cdot) \in \s H^{\infty}_\textnormal{w}(B(0,\eps');L(H_1))\\
   \eta_2(\cdot) &\coloneqq  \lim_{k\to\infty}\left(M_{22,n_k}(\cdot)\right)^{-1}  \in \s H^{\infty}_\textnormal{w}(E;L(H_2))\quad (E\subset\subset B(0,\eps')\setminus \{0\}) \\
   \eta_3(\cdot) &\coloneqq  \lim_{k\to\infty}M_{22,n_k}(\cdot)^{-1}M_{21,n_k}(\cdot)  \in \s H^{\infty}_\textnormal{w}(B(0,\eps');L(H_1,H_2))\text{ and } \\
   \eta_4(\cdot) &\coloneqq  \lim_{k\to\infty}M_{12,n_k}(\cdot)M_{22,n_k}(\cdot)^{-1} \in \s H^{\infty}_\textnormal{w}(B(0,\eps');L(H_2,H_1)). 
\end{align*}
Moreover, $R(N(0))=R(M_{1}(0))$.
\end{Sa}
\begin{proof} By Theorem \ref{Th:final} (applied to $M=M_n$ and the sequence $N$ just the constant sequence consisting of $M_n$ as every entry) there exist $\eps',r',c'>0$ such that, for all $n\in\N$, 
\begin{multline*}
   \begin{pmatrix}1 & -M_{12,n}(\cdot)M_{22,n}(\cdot)^{-1}\\ 0&1    
   \end{pmatrix}\begin{pmatrix}M_{11,n}(\cdot)&M_{12,n}(\cdot)\\M_{21,n}(\cdot)&M_{22,n}(\cdot)    
   \end{pmatrix}\begin{pmatrix}1 & 0\\-M_{22,n}(\cdot)^{-1}M_{21,n}(\cdot)&1    
   \end{pmatrix}\\
=\begin{pmatrix} M_{11,n}(\cdot)- M_{12,n}(\cdot)M_{22,n}(\cdot)^{-1}M_{21,n}(\cdot) & 0 \\ 0 & M_{22,n}(\cdot)\end{pmatrix}\\
   \in \s H^\infty(B(0,\eps');L(H_1\oplus H_2))\cap \s H^{\infty,c'}(B(r',r');L(H_1\oplus H_2)).
\end{multline*}
Let $\nu>1/(2r')$. By Theorem \ref{Thm: seq_comp} and Theorem \ref{Th: 0-anal-ode}, we may choose convergent subsequences of the material law sequences
\begin{align*}
   (\mu_{1,n})_n &\coloneqq  (M_{11,n}(\cdot)- M_{12,n}(\cdot)M_{22,n}(\cdot)^{-1}M_{21,n}(\cdot))_n \\
   (\mu_{2,n})_n &\coloneqq  (M_{22,n}(\cdot)^{-1})_n \\
   (\mu_{3,n})_n &\coloneqq  (M_{22,n}(\cdot)^{-1}M_{21,n}(\cdot))_n \text{ and } \\
   (\mu_{4,n})_n &\coloneqq  (M_{12,n}(\cdot)M_{22,n}(\cdot)^{-1})_n.
\end{align*}
We will use the same index for the subsequences and denote the respective limits by $\eta_1,\eta_2,\eta_3$ and $\eta_4$. Using the representation from Theorem \ref{Sa:2times2yields4times4}, we get with the help of Theorem \ref{Th:final}: 
\begin{align*}
(G_1 \oplus \{0\}) \oplus ( G_3 \oplus \{0\}) & = R\left( \begin{pmatrix} \begin{pmatrix}M_{11,n}^{(0)} & 0\\ 0& 0\end{pmatrix}  &
\begin{pmatrix} M_{13,n}^{(0)} & 0 \\  0         & 0 \end{pmatrix} \\
\begin{pmatrix}    M_{31,n}^{(0)} & 0 \\ 0&0\end{pmatrix} & \begin{pmatrix} M_{33,n}^{(0)} & 0\\   0         & 0 \end{pmatrix} \end{pmatrix}\right)\\
& =R\left(\begin{pmatrix}M_{11,n}(0)&M_{12,n}(0)\\M_{21,n}(0)&M_{22,n}(0)    
   \end{pmatrix}\right)\\
& = R\left( \begin{pmatrix} M_{11,n}(0)- M_{12,n}(0)M_{22,n}(0)^{-1}M_{21,n}(0) & 0 \\ 0 & M_{22,n}(0)\end{pmatrix}\right)\\
& = R\left( M_{11,n}(0)- M_{12,n}(0)M_{22,n}(0)^{-1}M_{21,n}(0)\right)\oplus R\left(M_{22,n}(0)\right).
\end{align*}
Now, $M_{11,n}(0)- M_{12,n}(0)M_{22,n}(0)^{-1}M_{21,n}(0)$ is strictly positive on $G_1=R(M_{11,1}(0))$ and $M_{22,n}(0)$ is strictly positive on $G_3=R(M_{22,1}(0))$ uniformly in $n$. Hence, we deduce that $R(\eta_{1}(0))=R(M_{11,1}(0))$ and, from Theorem \ref{Th: 0-anal-ode}, that $R(\eta_{2}(\cdot)^{-1}(0))=R(M_{22,1}(0))$. Let $(f_1,f_2)\in H_{\nu,-1}(\R;H_1\oplus H_2)$ and for $n\in\N$, let $(u_{1,n},u_{2,n}) \in H_{\nu,-1}(\R;H_1\oplus H_2)$ be the unique solution of
\[
   \partial_t \begin{pmatrix} M_{11,n}(\partial_t^{-1}) & M_{12,n}(\partial_t^{-1}) \\ M_{21,n}(\partial_t^{-1}) & M_{22,n}(\partial_t^{-1}) \end{pmatrix}\begin{pmatrix} u_{1,n} \\ u_{2,n} \end{pmatrix}  + 
  \begin{pmatrix} A & 0 \\ 0 & 0 \end{pmatrix} \begin{pmatrix} u_{1,n} \\ u_{2,n} \end{pmatrix} = \begin{pmatrix} f_1 \\ f_2 \end{pmatrix}. 
\]
Multiplying this equation by $\begin{pmatrix} 1 & -M_{12,n}(\partial_t^{-1})M_{22,n}(\partial_t^{-1})^{-1} \\ 0 & (\partial_tM_{22,n}(\partial_t^{-1}))^{-1} \end{pmatrix}$, we obtain
\[
    \begin{pmatrix} \partial_t\mu_{1,n}(\partial_t^{-1}) & 0 \\ \mu_{3,n}(\partial_t^{-1}) & 1 \end{pmatrix}\begin{pmatrix} u_{1,n} \\ u_{2,n} \end{pmatrix}  + 
  \begin{pmatrix} A & 0 \\ 0 & 0 \end{pmatrix} \begin{pmatrix} u_{1,n} \\ u_{2,n} \end{pmatrix} = \begin{pmatrix} f_1-\mu_{4,n}(\partial_t^{-1})f_2 \\ \mu_{2,n}(\partial_t^{-1})\partial_t^{-1} f_2 \end{pmatrix}.
\]
Thus,
\[
   \begin{pmatrix}
    u_{1,n} \\ u_{2,n} 
   \end{pmatrix} = \begin{pmatrix} (\partial_t\mu_{1,n}(\partial_t^{-1})+A)^{-1}(f_1-\mu_{4,n}(\partial_t^{-1})f_2) \\
                                   - \mu_{3,n}(\partial_t^{-1})u_{1,n}+\mu_{2,n}(\partial_t^{-1})\partial_t^{-1} f_2\end{pmatrix}.
\]
Lemma \ref{Le: loc_uni_weak} ensures that $(\mu_{4,n}(\partial_t^{-1})f_2)_n$ weakly converges to $\eta_4(\partial_t^{-1})f_2$. Thus, by 
 Theorem \ref{Thm:(P1)}, we deduce that $(u_{1,n})_n$ weakly converges to $(\partial_t\eta_1(\partial_t^{-1})+A)^{-1}(f_1-\eta_4(\partial_t^{-1})f_2)=:v_1$. 
Moreover, $(\partial_t^{-3}u_{1,n})_n$ converges pointwise to $\partial_t^{-3}v_1$. Using the equality 
\[
  u_{2,n}=-\mu_{3,n}(\partial_t^{-1}) u_{1,n} +\mu_{2,n}(\partial_t^{-1})\partial_t^{-1}f_2\in H_{\nu,-1}(\R;H_2), 
\]
 we deduce, with the help of Corollary \ref{Cor:W-Strong-Product} for the first term on the right-hand side and Theorem \ref{Th: 0-anal-ode} for the second term, that
\[
   u_{2,n} \rightharpoonup v_2\coloneqq -\eta_{3}(\partial_t^{-1}) v_{1} +\eta_{2}(\partial_t^{-1})\partial_t^{-1}f_2 \in H_{\nu,-1}(\R;H_2)
\]
as $n\to\infty$. We arrive at the limit system
\[
    \begin{pmatrix} \partial_t\eta_{1}(\partial_t^{-1}) & 0 \\ \eta_{3}(\partial_t^{-1}) & 1 \end{pmatrix}\begin{pmatrix} v_{1} \\ v_{2} \end{pmatrix}  + 
  \begin{pmatrix} A & 0 \\ 0 & 0 \end{pmatrix} \begin{pmatrix} v_{1} \\ v_{2} \end{pmatrix} = \begin{pmatrix} f_1-\eta_{4}(\partial_t^{-1})f_2 \\ \eta_{2}(\partial_t^{-1})\partial_t^{-1}f_2 \end{pmatrix}.
\]
Multiplying this equation by $\begin{pmatrix} 1 & \partial_t \eta_{4}(\partial_t^{-1})\eta_2(\partial_t^{-1})^{-1} \\ 0 & \eta_2(\partial_t^{-1})^{-1}\partial_t \end{pmatrix}$, we obtain
\begin{multline*}
    \partial_t \begin{pmatrix} \eta_{1}(\partial_t^{-1}) + \eta_{4}(\partial_t^{-1})\eta_2(\partial_t^{-1})^{-1}\eta_{3}(\partial_t^{-1}) & \eta_{4}(\partial_t^{-1})\eta_2(\partial_t^{-1})^{-1} 
\\ \eta_2(\partial_t^{-1})^{-1}\eta_{3}(\partial_t^{-1}) & \eta_2(\partial_t^{-1})^{-1} \end{pmatrix}\begin{pmatrix} v_{1} \\ v_{2} \end{pmatrix}  +\\ 
  \begin{pmatrix} A & 0 \\ 0 & 0 \end{pmatrix} \begin{pmatrix} v_{1} \\ v_{2} \end{pmatrix} =  \begin{pmatrix} f_1\\ f_2 \end{pmatrix}. 
\end{multline*}
Next, we consider the operator 
\[
   N(\cdot) = \begin{pmatrix} \eta_{1}(\cdot) + \eta_{4}(\cdot)\eta_2(\cdot)^{-1}\eta_{3}(\cdot) & \eta_{4}(\cdot)\eta_2(\cdot)^{-1} 
\\ \eta_2(\cdot)^{-1}\eta_{3}(\cdot) & \eta_2(\cdot)^{-1} \end{pmatrix} = \begin{pmatrix} 1 & \eta_4(\cdot) \\ 0 & 1 \end{pmatrix} 
\begin{pmatrix} \eta_1(\cdot) & 0 \\ 0 & \eta_2(\cdot)^{-1} \end{pmatrix} \begin{pmatrix} 1 & 0 \\ \eta_3(\cdot) & 1 \end{pmatrix}.
\]
By Theorem \ref{Th: 0-anal-ode}, we deduce that $\eta_2(\cdot)^{-1}$ is a $(c'')$-material law with strictly positive zeroth order term on the range of $M_{22,1}(0)$ for some $c''>0$. Moreover, $\eta_1$ is a $(c')$-material law by Theorem \ref{Thm:(P1)}. Hence, using Theorem \ref{Th:final}, we deduce the existence of $\eps'',r'',c'''>0$ such that
$N \in \s H^{\infty,c'''}(B(r'',r'');L(H_1\oplus H_2))\cap \s H^{\infty}(B(0,\eps'');L(H_1\oplus H_2))$. 
\end{proof}
\begin{rems}
  In Theorem \ref{Thm:(P2)}, we have proved that $0$-analytic material laws lead to $0$-analytic material laws after the homogenization process. Hence, it cannot be expected that the homogenized material law contains fractional derivatives with respect to time or explicit delay terms: Indeed, these operators cannot be represented as material laws, which are analytic in $0$, see e.g.~\cite[pp. 448 (a),(c)]{Picard} or \cite{Drrerfrac}. By Theorem \ref{Th: SolTh} we see that the limit equation is also well-posed and causal. The assertion concerning the range of the material law $N$ may be interpreted as ``the main physical phenomenon remains unchanged under the homogenization process'': A clarification of the latter statement is in order. One difference between the wave equation and the heat equation written as a first order system as in \cite[Example 1.4.6]{Waurick2011} or \cite[Example 3.2]{Waurick2010} is the range of the zeroth order term in the material law. More precisely, let $\Omega\subseteqq \mathbb{R}^N$ open, $\kappa\in L^\infty(\Omega)^{N\times N}$ such that $\kappa^{-1}\in L^\infty(\Omega)^{N\times N}$. For smooth and compactly supported $f$ and $g$ we shall rewrite the wave equation 
  \[
      \partial_t^2u - \diverg \kappa \interior{\grad} u = f
  \]
  and the heat equation
  \[
    \partial_t \theta - \diverg \kappa \interior{\grad} \theta = g
  \]
  as first order systems. Setting $v\coloneqq \partial_t u, w\coloneqq -\kappa \interior\grad u$; $q\coloneqq -\kappa \interior\grad\theta$, we get
  \[
     \Big(\partial_t \begin{pmatrix} 1 & 0 \\ 0 & \kappa^{-1} \end{pmatrix} + \begin{pmatrix} 0 & \diverg \\  \interior{\grad} &0 \end{pmatrix}\Big)\begin{pmatrix} v \\ w \end{pmatrix} =\begin{pmatrix} f \\ 0 \end{pmatrix}
  \]
  and
  \[
     \Big(\partial_t \begin{pmatrix} 1 & 0 \\ 0 & 0 \end{pmatrix}+\begin{pmatrix} 0 & 0 \\ 0 & \kappa^{-1} \end{pmatrix} + \begin{pmatrix} 0 & \diverg \\  \interior{\grad} &0 \end{pmatrix}\Big)\begin{pmatrix} \theta \\ q \end{pmatrix} =\begin{pmatrix} g \\ 0 \end{pmatrix},
  \]
  respectively. Therefore, the corresponding material laws read
  \[
     M_{\text{wave}}(\partial_t^{-1}) = \begin{pmatrix} 1 & 0 \\ 0 & \kappa^{-1} \end{pmatrix}
  \]
  and
  \[
     M_{\text{heat}}(\partial_t^{-1}) = \begin{pmatrix} 1 & 0 \\ 0 & 0 \end{pmatrix}+ \partial_t^{-1}\begin{pmatrix} 0 & 0 \\ 0 & \kappa^{-1} \end{pmatrix}.
  \]
  So, $M_{\text{wave}}(0)$ is onto. But the range of $M_{\text{heat}}(0)$ coincides with $L^2(\Omega)\oplus \{0\}\subseteqq L^2(\Omega)\oplus L^2(\Omega)^N$. According to Theorem \ref{Thm:(P2)}, this property remains unchanged due to the homogenization process.
\end{rems}

\begin{Fo}\label{Co: Askew} Let $H$ be a separable Hilbert space, $\eps,c,r>0$, $A:D(A)\subseteqq H\to H$ skew-selfadjoint. Denote by $P: H\to N(A)^\bot$, $Q:H\to N(A)$ the orthogonal projections onto the 
respective spaces $N(A)^\bot$ and $N(A)$. Assume that the operator $A$ has the \emph{$(NC)$-property}, that is, $(D(PAP^*),\abs{\cdot}_{PAP^*})\hookrightarrow (H,\abs{\cdot}_H)$ is compact. Let $(M_n)_n$ be a bounded sequence in $\s H^\infty(B(0,\eps);L(H))\cap \s H^{\infty,c}(B(r,r);L(H))$ with $M_n(0)\geqq c$ on $R(M_n(0))=R(M_1(0))$ for all $n\in\N$. Denote by $q_2: H\to N(M_1(0))\cap N(A)^\bot$, $q_4: H\to N(M_1(0))\cap N(A)$ the canonical orthogonal projections and assume 
\begin{equation}\label{eq:comp_cond}
  q_2 M_n'(0)q_4^*(q_4  M_n'(0)q_4^*)^{-1}=q_2 M_n'(0)^*q_4^*(q_4 M_n'(0)^*q_4^*)^{-1} \text{ for all } n\in\N. 
\end{equation} 
 Then there exists $\nu_0\geqq 0,\eps',c'>0$ and $(n_k)_k$ a strictly monotone sequence of positive integers
 such that for all $\nu>\nu_0$, the sequence 
\[
  (\partial_t M_{n_k}(\partial_t^{-1})+A)_k 
\]
 $G$-converges to 
\begin{multline*}
   \partial_t \left(P^*\eta_{1}(\partial_t^{-1})P + P^*\eta_{4}(\partial_t^{-1})\eta_2(\partial_t^{-1})^{-1}\eta_{3}(\partial_t^{-1})P +P^*\eta_{4}(\partial_t^{-1})\eta_2(\partial_t^{-1})^{-1}Q \right.\\ 
\left.+ Q^*\eta_2(\partial_t^{-1})^{-1}\eta_{3}(\partial_t^{-1})P + Q^*\eta_2(\partial_t^{-1})^{-1}Q \right)+ A  
\end{multline*}
in $H_{\nu,-1}(\R;H)$, where\footnote{The limits are computed in the way similar to Theorem \ref{Thm:(P2)} with $H_1=N(A)^\bot$ and $H_2=N(A)$.}
\begin{align*}
   \eta_1(\cdot) &\coloneqq  \lim_{k\to\infty}PM_{n_k}(\cdot)P^*- PM_{n_k}(\cdot)Q^*(QM_{n_k}(\cdot)Q^*)^{-1}QM_{n_k}(\cdot)P^* \\
   \eta_2(\cdot) &\coloneqq  \lim_{k\to\infty}(QM_{n_k}(\cdot)Q^*)^{-1} \\
   \eta_3(\cdot) &\coloneqq  \lim_{k\to\infty}(QM_{n_k}(\cdot)Q^*)^{-1}QM_{n_k}(\cdot)P^* \text{ and } \\
   \eta_4(\cdot) &\coloneqq  \lim_{k\to\infty}(PM_{n_k}(\cdot)Q^*)(QM_{n_k}(\cdot)Q^*)^{-1}. 
\end{align*}
\end{Fo}
\begin{proof}
  The assertion follows by applying Theorem \ref{Thm:(P2)} to 
\[
   ((M_n)_n,\s A)=
 \left( \left(\begin{pmatrix} PM_n(\cdot)P^* & PM_n(\cdot)Q^* \\ QM_n(\cdot)P^* & QM_n(\cdot)Q^*\end{pmatrix}\right)_n, \begin{pmatrix} PAP^* & 0 \\ 0 & 0 \end{pmatrix}\right).\qedhere
\]
\end{proof}
\begin{rems} The compatibility condition \eqref{eq:comp_cond} may be hard to check in applications. However, there are some situations in which the Condition \eqref{eq:comp_cond} is trivially satisfied:
\begin{itemize}
 \item $A$ is one-to-one; then $N(A)=\{0\}$ and $q_4=0$.
 \item $R(M_1(0))\supseteqq N(A)^{\bot}$; then $N(M_1(0))\cap N(A)^\bot=\{0\}$ and $q_2=0$.
 \item $M_1(0)$ is onto; then the preceding condition is satisfied. We remark here that this condition was imposed in \cite[Theorem 2.3.14]{Waurick2011}. This condition corresponds to hyperbolic-type equations in applications.
 \item $M_n'(0)=M_n'(0)^*$; then $q_2 M_n'(0)q_4^*(q_4  M_n'(0)q_4^*)^{-1}=q_2 M_n'(0)^*q_4^*(q_4 M_n'(0)^*q_4^*)^{-1}$.
\end{itemize}
\end{rems}
We do not yet know whether the compatibility condition is optimal. We can however give some examples to show that the other assumptions in $(P2)$ are reasonable. The following example shows that without the requirement on $A$ to have the $(NC)$-property the limit equation can differ from the expressions given in Theorem \ref{Thm:(P2)} or Corollary \ref{Co: Askew}.
\begin{Bei}[Compactness assumption does not hold]\label{Ex: Count} Let $\nu,\eps>0$. Consider the mapping $a:\R \to\R$ given by
\[
   a(x) \coloneqq  \1_{[0,\frac12)}(x-k)+2\1_{[\frac12,1]}(x-k)
\]
for all $x\in  [k,k+1)$, where $k\in\Z$. By $a(n\cdot \hat m)\phi\coloneqq (x\mapsto a(nx)\phi(x))$ for $n\in\N$, we define the corresponding multiplication operator in $L_2(\R)$. Note that $a(x+k)=a(x)$ for all $x\in\R$ and $k\in\Z$.

Let $f\in H_{\nu,0}(\R; L_2(\R))$. We consider the evolutionary equation with $(M_n(\partial_t^{-1}))_n\coloneqq (\partial_t^{-1}a(n\cdot \hat m))_n$ and $A=i: L_2(\R)\to L_2(\R): \phi\mapsto i\phi$. Clearly, $N(A)=\{0\}$. By \cite[Theorem 4.5]{Waurick} or \cite[Theorem 1.5]{Dac}, we deduce that 
\[
   M_n\to \left(z\mapsto z\int_0^1 a(x)\dd x\right)=\left(z\mapsto \frac32 z\right) \in \s H_{\textnormal{w}}^\infty(B(0,\eps);L(L_2(\R)))
\]
as $n\to\infty$. If the assertion of Corollary \ref{Co: Askew} remains true in this case, then $(\partial_tM_n(\partial_t^{-1})+A)_n$ $G$-converges to $\frac32 + i$. For $n\in\N$, let $u_n\in H_{\nu,0}(\R;L_2(\R))$ be the unique solution of the equation
\begin{equation}\label{a_n}
  (\partial_tM_n(\partial_t^{-1})+A)u_n=(a(n\cdot\hat m )+i)u_n=f.
\end{equation}
Observe that by \cite[Theorem 1.5]{Dac}
\[
  u_n=(a(n\hat m)+i)^{-1}f\rightharpoonup \left(\int_0^1(a(x)+i)^{-1}\dd x\right) f=:u.
\]
as $n\to\infty$. We integrate
\[
  \int_0^1(a(x)+i)^{-1}\dd x =\frac12(1+i)^{-1}+\frac12(2+i)^{-1}.
\]
Inverting the latter equation yields
\[
  \left(\int_0^1(a(x)+i)^{-1}\dd x\right)^{-1} =\left(\frac12(1+i)^{-1}+\frac12(2+i)^{-1}\right)^{-1}=\frac{18}{13}+\frac{14}{13}i.
\]
Hence, $u$ satisfies
\[
  \left(\frac32 + i\right)u=f \text{ and }\left(\frac{18}{13}+\frac{14}{13}i\right)u=f,
\]
which of course is a contradiction.
\end{Bei}
In the next example, the uniform positive definiteness is violated.
\begin{Bei}[Uniform positive definiteness condition does not hold] Let $H=\C$, $\nu>0$ and, for $n\in\N$, let $M_n(\partial_t^{-1})=\partial_t^{-1}\frac{1}{n}$, $A=0$, $f\in H_{\nu,0}(\R)\setminus\{0\}$. For $n\in\N$, let  $u_n\in H_{\nu,0}(\R)$ be defined by
\[
 \partial_t M_n(\partial_t^{-1})u_n = \frac{1}{n}u_n = f.
\]
 Then $(u_n)_n$ is not relatively weakly compact and contains no weakly convergent subsequence.
\end{Bei}
In the final example, the range condition in $(P2)$ is violated. 
\begin{Bei}[Range condition does not hold] Let $H$ be an infinite-dimensional, separable Hilbert space. Let $(\phi_n)_n$ be a complete orthonormal system. For $n\in\N$ define $M_n(\partial_t^{-1})\coloneqq  \langle \phi_n,\cdot\rangle\phi_n+\partial_t^{-1}(1-\langle\phi_n,\cdot\rangle\phi_n)$. For the sequence $(M_n)_n$ the range condition in $(P2)$ (applied with $A=0$) is violated. Let $f\in H_{\nu,0}(\R;H)$, $\nu>0$. For $n\in\N$, let $u_n\in H_{\nu,0}(\R;H)$ be such that
 \[
    \partial_t M_n(\partial_t^{-1})u_n = \partial_t \langle \phi_n,u_n\rangle\phi_n+u_n - \langle \phi_n,u_n\rangle \phi_n =f.
 \]
It is easy to see that $(u_n)_n$ is bounded. Take the inner product of the last equation with $\phi_m$ for some $m\in \N$. If $n\in\N$ is larger than $m$ we arrive at
\[
   \langle u_n, \phi_m \rangle = \langle f,\phi_m\rangle,
\]
and we deduce that $(\partial_t M_n(\partial_t^{-1}))_n$ $G$-converges to $\partial_t\partial_t^{-1}=1$. This, however, does not yield a differential equation.
\end{Bei}

\section{Applications}\label{section: Ex}

We demonstrate the applicability of our main theorem to the mathematical models of some physical phenomena. For notational details, we refer to \cite[pp. 34 and p. 98]{Waurick2011} or to \cite[3.2 Examples]{Trostorff2011a}.

\subsection*{Thermodynamics} Let $\alpha,\beta\in \R$, $0<\alpha<\beta$, $\Omega\subseteqq \R^N$ open and bounded, $N\in\N$. Recall \cite[Definition 4.11]{CioDon}:
  \begin{multline*}
  M(\alpha,\beta,\Omega)\coloneqq  \big\{ \kappa \in L_\infty (\Omega)^{N\times N};\\ \Re \langle \kappa(x)\xi,\xi\rangle \geqq \alpha \abs{\xi}^2, \abs{\kappa(x)\xi}\leqq \beta\abs{\xi}, \xi\in \R^N, \text{ a.e. }x\in \Omega \big\}. \end{multline*}
Let $M_\text{self}(\alpha,\beta,\Omega)\coloneqq \{ \kappa\in M(\alpha,\beta,\Omega); \kappa \text{ selfadjoint a.e.}\}$. For $\kappa\in M(\alpha,\beta,\Omega)$ denote by $\kappa(\hat m)$ the  associated multiplication operator in $L_2(\Omega)^N$. Let $(\kappa_n)_n$ be a sequence in $M_\text{self}(\alpha,\beta,\Omega)$. Recall from Section \ref{section: disco} that a first order formulation of the heat equation with Dirichlet boundary conditions in the context of evolutionary equations introduced in \cite{PicPhy} is the following
\[
   \left(\partial_t \begin{pmatrix} 1 & 0 \\ 0 & 0 \end{pmatrix} + \begin{pmatrix} 0 & 0 \\ 0 & \kappa_n(\hat m)^{-1}  \end{pmatrix} + \begin{pmatrix} 0 & \diverg \\ {\interior{\grad}} & 0 \end{pmatrix}\right)\begin{pmatrix} u_{1,n} \\ u_{2,n}\end{pmatrix} = \begin{pmatrix} f_1 \\ f_2 \end{pmatrix}.
\]
We want to apply Corollary \ref{Co: Askew} with \[
A=\begin{pmatrix} 0 & \diverg \\ {\interior{\grad}} & 0 \end{pmatrix}
                                                \]
and 
\[
 (M_n(\partial_t^{-1}))_n \coloneqq  \left(\begin{pmatrix} 1 & 0 \\ 0 & 0 \end{pmatrix} + \partial_t^{-1}\begin{pmatrix} 0 & 0 \\ 0 & \kappa_n(\hat m)^{-1}  \end{pmatrix}\right)_n. 
\]
 The compactness condition on the operator $\begin{pmatrix} 0 & \diverg \\ {\interior{\grad}} & 0 \end{pmatrix}$ has already been established e.g.\ in \cite[Remark 3.2.2]{Waurick2011} or \cite[the end of the proof of Theorem 4.3]{Waurick2012} and the sequence $(\kappa_n(\hat m)^{-1})_n$ is a sequence of selfadjoint operators. Since $M_n(0)=M_1(0)=\begin{pmatrix} 1 & 0 \\ 0& 0 \end{pmatrix}$ for all $n\in\N$, the range condition is satisfied. Thus Corollary \ref{Co: Askew} applies. The sequence $(M_n)_n$ could be replaced by some convolution terms. Moreover it should be noted that the case of not necessarily symmetric $\kappa_n$'s has been considered in \cite{Waurick2012}. There, however, a second-order formulation was used which, for more general material laws, may not be available. The homogenized equations are derived in Section \ref{section: disco}, see equation \eqref{eq:phys_limit}.

\subsection*{Electromagnetism}
Let $\Omega\subseteqq \R^3$ open. The general form of Maxwell's equations in bi-anisotropic dissipative media used in \cite{BarStr} is
\[
   \left(\partial_t\begin{pmatrix} \eps & \gamma \\ \gamma^* & \mu \end{pmatrix} + \begin{pmatrix} \sigma_{11}* & \sigma_{12}* \\ \sigma_{21}* & \sigma_{22}* \end{pmatrix}+\begin{pmatrix} 0 & -\curl \\ {\interior{\curl}} & 0 \end{pmatrix}\right)\begin{pmatrix} E \\ H \end{pmatrix}= \begin{pmatrix} J\\ 0\end{pmatrix} 
\] with a $(c)$-material law (see Section \ref{section: set} for a definition) \[M(\partial_t^{-1})\coloneqq \begin{pmatrix} \eps & \gamma \\ \gamma^* & \mu \end{pmatrix} + \partial_t^{-1}\begin{pmatrix} \sigma_{11}* & \sigma_{12}* \\ \sigma_{21}* & \sigma_{22}* \end{pmatrix}\] for a $c>0$. Here, $\sigma_{jk}$ are $L(L_2(\Omega)^3)$-valued functions on $\R$, vanishing on $\R_{<0}$ and being such that 
the temporal convolutions $\sigma_{jk}*$ yield $0$-analytic material laws, $j,k\in\{1,2\}$. Moreover, the operator $\begin{pmatrix} \eps & \gamma \\ \gamma^* & \mu \end{pmatrix}$ is assumed to be selfadjoint and strictly positive definite in $L_2(\Omega)^6$. We emphasize here that the convolution kernels may also take values in the linear operators on $L_2(\Omega)^3$, which are not representable as multiplication operators, thus, in this way, generalizing the assumptions in \cite{BarStr}.
Now, consider a sequence of $(c)$-material laws $(M_n)_n$ of the above form with non-singular, strictly positive zeroth order term: $\begin{pmatrix} \eps_n & \gamma_n \\ \gamma^*_n & \mu_n \end{pmatrix}\geqq d>0$ for all $n\in \N$. Then, Corollary \ref{Co: Askew} applies if we assume $\Omega$ to be bounded and to satisfy suitable smoothness assumptions on the boundary, see e.g.~\cite{Picard1984,Picard2001,Witsch1993} (or (also for possibly other boundary conditions) \cite{BPS16,Costabel1990,Jochmann1997,Kuhn1999,Weber1980,Weck1974}). Indeed, the range condition is satisfied since $M_n(0)=\begin{pmatrix} \eps_n & \gamma_n \\ \gamma^*_n & \mu_n \end{pmatrix}$ is onto for all $n\in\N$  and, since $N(M_n(0))=\{0\}$, the compatibility condition \eqref{eq:comp_cond} also follows. Note that the homogenized equations are more complicated than in the case of the heat equation. This is due to the fact that both the off-diagonal operators in $A=\begin{pmatrix} 0 & -\curl \\ {\interior{\curl}} & 0 \end{pmatrix}$ have an infinite-dimensional nullspace. In case of the heat equation with Dirichlet boundary conditions, we have $A=\begin{pmatrix} 0 & \diverg \\ {\interior{\grad}} & 0 \end{pmatrix}$, where the nullspace of ${\interior{\grad}}$ is trivial.

To illustrate the versatility and applicability of Theorem \ref{Thm:(P2)} we show how our methods apply to the equations of thermopiezoelectricity.

\subsection*{Thermopiezoelectricity}
We assume $\Omega\subseteqq \R^3$ to be open and bounded. The equations of thermopiezoelectricity describe the interconnected effects of elasticity, thermodynamics and electro-magnetism. The set $\Omega$ models a body in its non deformed state. We recall the formulation as in \cite[6.3.3, p. 457]{Picard}, where the model given in \cite{Mind1974} is discussed, see also \cite{Mulholland2015}. The unknowns of the system are the time-derivative of the displacement field $v$, the stress tensor $T$, the electric and magnetic field $E$ and $H$ as well as the heat distributions $\theta$ and the heat flux $q$. Recall the spatial derivative operators from the introduction and define $\Diverg$ as the negative adjoint of $\interior\Grad$, the symmetrized gradient with homogeneous Dirichlet boundary conditions. The equations read as follows
\begin{multline*}
   \left( \partial_t \begin{pmatrix}
         \rho_0  & 0 & 0   &   0    &  0 & 0\\
        0 & C^{-1}  & C^{-1}d  &   0  & C^{-1}\lambda  &  0 \\
        0  &  d^*C^{-1} & \eps+d^*C^{-1}d  & 0&  p+d^*C^{-1}\lambda  & 0   \\
        0  &  0 & 0&\mu&  0  & 0   \\
        0  &  \lambda^*C^{-1} & p^*+\lambda^*C^{-1}d       &0&  \alpha+\lambda^* C^{-1}\lambda  & 0  \\
        0  &  0 & 0       &0&  0  & q_0+q_1(\alpha+\kappa\partial_t)^{-1} 
       \end{pmatrix}\right.\\ \left. + \begin{pmatrix}
         0 & \Diverg & 0   &    0  &  0 & 0\\
        \interior\Grad & 0  & 0   &    0  &  0 & 0\\
        0  &  0 & 0  & -\curl&   0  &  0 \\
        0  &  0 & {\interior{\curl}}&0&   0  &  0 \\
        0  &  0 & 0       &0&   0  &  \div \\
        0  &  0 & 0       &0&   {\interior{\grad}}  &  0 
       \end{pmatrix}\right)\begin{pmatrix} v \\ T \\ E \\ H\\ \theta \\ q \end{pmatrix} = \begin{pmatrix} f \\ 0 \\ J \\ 0 \\ g \\ 0 \end{pmatrix},
\end{multline*}
where $\rho_0,C,\eps,\mu,q_0,q_1,\alpha,\kappa,d,\lambda,p$ are bounded linear operators in appropriate $L_2(\Omega)$-spaces. 
To frame the latter system into the general context of this exposition, we find that
\[
   A = \begin{pmatrix}
         0 & \Diverg & 0   &    0  &  0 & 0\\
        \interior\Grad & 0  & 0   &    0  &  0 & 0\\
        0  &  0 & 0  & -\curl&   0  &  0 \\
        0  &  0 & {\interior{\curl}}&0&   0  &  0 \\
        0  &  0 & 0       &0&   0  &  \div \\
        0  &  0 & 0       &0&   {\interior{\grad}}  &  0 
       \end{pmatrix}.
\]
 The material law is given by 
 \[
  M(\partial_t^{-1}) = \begin{pmatrix}
         \rho_0  & 0 & 0   &   0    &  0 & 0\\
        0 & C^{-1}  & C^{-1}d  &   0  & C^{-1}\lambda  &  0 \\
        0  &  d^*C^{-1} & \eps+d^*C^{-1}d  & 0&  p+d^*C^{-1}\lambda  & 0   \\
        0  &  0 & 0&\mu&  0  & 0   \\
        0  &  \lambda^*C^{-1} & p^*+\lambda^*C^{-1}d       &0&  \alpha+\lambda^* C^{-1}\lambda  & 0  \\
        0  &  0 & 0       &0&  0  & q_0+q_1(\alpha+\kappa\partial_t)^{-1} 
       \end{pmatrix},
 \]
 $M$ is a $(c)$-material law, if we assume that one of the following conditions is satisfied
\ben
  \item\label{tpe1} $\rho_0,C,\eps,\mu,q_0,\alpha-p^*\eps^{-1}p,\kappa$ are selfadjoint and strictly positive,
  \item\label{tpe2} $\rho_0,C,\eps,\mu,q_1,\alpha-p^*\eps^{-1}p,\kappa$ are selfadjoint, strictly positive, $q_1\kappa^{-1}=\kappa^{-1}q_1$ and $q_0=0$.
\een
Indeed, the material law can be written as a block operator matrix in the form
\[
  M(\partial_t^{-1})= \begin{pmatrix}
                       M_{11} & 0 \\ 0 & M_{22}(\partial_t^{-1})
                      \end{pmatrix},
\]
where 
\begin{align*}
  M_{11}&= \begin{pmatrix}
         \rho_0  & 0 & 0   &   0    &  0 \\
        0 & C^{-1}  & C^{-1}d  &   0  & C^{-1}\lambda   \\
        0  &  d^*C^{-1} & \eps+d^*C^{-1}d  & 0&  p+d^*C^{-1}\lambda     \\
        0  &  0 & 0&\mu&  0     \\
        0  &  \lambda^*C^{-1} & p^*+\lambda^*C^{-1}d       &0&  \alpha+\lambda^* C^{-1}\lambda
       \end{pmatrix}\\&= \begin{pmatrix}
        1  & 0  & 0   &   0    &  0 \\
        0  & 1  & 0  &   0  & 0   \\
        0  & d^*& 1  & 0&  0     \\
        0  &  0 & 0  & 1&  0     \\
        0  &  \lambda^* & p^*\eps^{-1}    &0&  1
       \end{pmatrix}\begin{pmatrix}
         \rho_0  & 0 & 0   &   0    &  0 \\
        0 & C^{-1}  & 0  &   0  & 0   \\
        0  & 0 & \eps& 0& 0     \\
        0  &  0 & 0&\mu&  0     \\
        0  &  0 & 0       &0&  \alpha - p^*\eps^{-1}p
       \end{pmatrix}\begin{pmatrix}
         1  & 0 & 0   &   0    &  0 \\
        0 & 1  & d  &   0  & \lambda   \\
        0  &  0 & 1  & 0&  \eps^{-1}p     \\
        0  &  0 & 0&1&  0     \\
        0  &  0 & 0       &0&  1
       \end{pmatrix}
\end{align*} is strictly positive and, by choosing $\nu>0$ sufficiently large such that $\Abs{\partial_t^{-1}}$ becomes small enough,
\[
  M_{22}(\partial_t^{-1})=q_0 + q_1(\alpha+\kappa\partial_t)^{-1}=q_0+q_1\kappa^{-1}\partial_t^{-1} + \sum_{n=1}^\infty (-1)^n \partial_t^{-n-1}(\kappa^{-1}\alpha)^n\kappa^{-1} 
\]
is such that either $M_{22}(0)$ or $M_{22}'(0)$ is strictly positive. Thus, $M$ is a $(c)$-material law for some $c>0$.

Considering a sequence $(M_n)_n$ of such material laws with the boundedness and uniform positive definiteness assumptions from Corollary \ref{Co: Askew}, one may derive a homogenization result for these equations. We will not do this explicitly here. However, in order to satisfy the range condition, one has to assume that all entries of the material law sequence satisfy either condition \eqref{tpe1} or \eqref{tpe2}. To deduce that $A$ has the $(NC)$-property, we have to impose suitable geometric requirements on $\Omega$ as in the previous example. 

We refer the interested reader to more examples of first order formulations of standard evolutionary equations in mathematical physics to \cite{PicPhy,Picard}. With these formulations it is then rather straightforward to see when and how our homogenization result applies. 

\section{Auxiliary results on $0$-analytic material laws}\label{section: aux}

In this section, we provide the remaining results needed in Section \ref{section: hom}. Our main concern will be the discussion of \emph{$0$-analytic} material laws, that is, material laws that are analytic at $0\in\C$, cf.~\cite[Section 3.3]{PicPhy}. To establish Theorem \ref{Thm:(P2)} similarity transformations of $0$-analytic material laws have to be discussed, where our main interest is to show that under any of these similarity transformations a $(c)$-material law transforms into a $(c')$-material law for suitable $c'>0$. In order to achieve the main goal of this section, Theorem \ref{Th:final}, some technical results are required. We start with a fact concerning Hardy space functions. 

\begin{Le}\label{Le: CIF} Let $X$ be a Banach space, $\eps>0$, $\mu(\cdot)=\sum_{n=0}^\infty (\cdot)^{n}\mu_{n} \in \s H^\infty(B(0,\eps);X)$. Then for all $k,n\in\N_0$ we have
\ben
   \item $\Abs{\mu_n}\leqq \Abs{\mu}_\infty\left(\frac{2}{\eps}\right)^n$
   \item $\Abs{\sum_{n=k}^\infty z^{n-k}\mu_n  }\leqq 2\Abs{\mu}_\infty \left( \frac{2}{\eps}\right)^{k}$ for all $z\in B\left( 0,\frac{\eps}{4}\right)$.  
\een 
 \end{Le}
\begin{proof}
 The first assertion follows immediately from Cauchy's integral formula (integrate over a circle around $0$ with radius $\eps/2$) and the second is a straightforward consequence of the first.
\end{proof}
With these estimates, we can establish some properties of $0$-analytic material laws. Recall that the inner products discussed here are linear in the second and conjugate linear in the first component.
\begin{Prop}\label{Prop: M0+M1} Let $H$ be a Hilbert space, $\eps,c>0$, $0<r<\eps/2$. Let $M$ be a material law in $\s H^{\infty}(B(0,\eps);L(H))\cap \s H^{\infty,c}(B(r,r);L(H))$. Then $M(0)$ is selfadjoint. For $\phi=\phi_1 \oplus \phi_2\in \overline{R(M(0))}\oplus N(M(0))$, the inequalities
\[
    \langle M(0)\phi_1,\phi_1\rangle \geqq 0 \text{ and } \langle \Re M'(0)\phi_2,\phi_2\rangle \geqq c\langle\phi_2,\phi_2\rangle
\] 
hold. If, in addition, $R(M(0))\subseteqq H$ is closed, there exists $d>0$, such that for $\phi_1\in R(M(0))$ we have
\[
 \langle M(0)\phi_1,\phi_1\rangle \geqq d\langle \phi_1,\phi_1\rangle.
\]\end{Prop}
\begin{proof}
  We expand $M$ into a power series about $0$: $M(z)=\sum_{n=0}^\infty z^n M_n$ for $z\in B(0,\eps)$ and suitable $(M_n)_n$ in $L(H)$. Then $M(0)=M_0$ and $M'(0)=M_1$. For $\phi\in H$ define $x_\phi\coloneqq  \Im \langle M(0)\phi,\phi\rangle$ and $y_\phi\coloneqq  \Re \langle M(0)\phi,\phi\rangle$. It is easy to see that $T: B(r,r) \to \R_{>\frac{1}{2r}}+i\R: z\mapsto z^{-1}$ is homeomorphic. Thus, for $z\in B(r,r)$ with $z_1\coloneqq \Im T(z)$, $z_2\coloneqq \Re T(z)$, we have
\begin{align*}
  c\langle \phi,\phi\rangle &\leqq \Re \langle\phi, z^{-1} M(z)\phi\rangle\\
                            &= \Re  (i z_1 +z_2) (ix_\phi +y_\phi)+\Re \langle \phi, \sum_{n=1}^\infty z^{n-1} M_n \phi \rangle\\
                            &= -z_1x_\phi + z_2 y_\phi +\Re \langle \phi, \sum_{n=1}^\infty z^{n-1} M_n \phi \rangle.
\end{align*}
The left-hand side is non-negative. The last term on the right-hand side is bounded for $z\to 0$. Moreover, since $T$ is bijective (in particular, for every $z_2$ the values of $z_1$ range over the whole real axis), it follows that $x_\phi=0$. Thus, we arrive at
\[
  c\langle \phi,\phi\rangle\leqq  z_2 y_\phi +\Re \langle \phi, \sum_{n=1}^\infty z^{n-1} M_n \phi \rangle.
\]
Now, since $z_2$ can be chosen arbitrarily large, while the second term of the right-hand side remains bounded, it follows that $y_\phi\geqq 0$. Thus, for every $\phi\in H$, we deduce that $\langle \phi, M(0)\phi\rangle\geqq 0$. Since $M(0)$ is a bounded operator in the complex Hilbert space $H$, the operator $M(0)$ is selfadjoint and positive (semi-)definite and therefore $H=\overline{R(M(0))}\oplus N(M(0))$. Let $\phi\in N(M(0))$. Then 
for $\eps/2>\eta>0$, 
\begin{align*}
   c\abs{ \phi}^2&\leqq \Re \langle \phi, \sum_{k=1}^\infty \eta^{k-1} M_k \phi \rangle\\
                 &= \Re \langle \phi, M'(0)\phi\rangle + \eta\Re \langle \phi, \sum_{k=2}^\infty \eta^{k-2} M_k \phi \rangle. 
\end{align*}
If we let $\eta\to 0+$, we get that $\Re \langle \phi, M'(0)\phi\rangle\geqq c\abs{ \phi}^2$. Now, $M(0)$ is invariant on its range and the restriction of $M(0)$ to its range is one-to-one. Thus, if $R(M(0))$ is closed, the closed graph theorem implies that $M(0):R(M(0))\to R(M(0))$ is continuously invertible. By the spectral theorem for continuous and selfadjoint operators it follows that $M(0)$ is strictly positive on its range.
\end{proof}
\begin{rems}\label{rem:conv}
 We shall note here that the converse of Proposition \ref{Prop: M0+M1} is also true in the following sense: Let $M \in \s H^\infty(B(0,\eps);L(H))$  be such that $M(0)=M(0)^*$. Assume that there exist $d,c>0$ such that for all $\phi_1\in R(M(0)), \phi_2\in N(M(0))$
\[
   \langle M(0)\phi_1,\phi_1\rangle \geqq d\langle \phi_1,\phi_1\rangle \text{ and } \langle \Re M'(0)\phi_2,\phi_2\rangle \geqq c\langle\phi_2,\phi_2\rangle.
\]
Then $R(M(0))\subseteqq H$ is closed and for $0<r\leqq \frac{1}{2\max\{\nu_1,{\hat\delta}^{-1}\}}$, $M\in \s H^{\infty,c/3}(B(r,r);L(H))$, cf.~\cite[Lemma 2.3]{Trostorff2011} or \cite[Remark 6.2.7]{Picard}, where $\nu_1\coloneqq \frac{1}{d}\left(\frac{2c}{3}+\frac{3}{c}\left( \Abs{M}_\infty\frac{2}{\eps}\right)^2+\frac{2}{\eps}\Abs{M}_\infty\right)$ and $\hat\delta\coloneqq \min\{\Abs{M}_\infty^{-1} \left(\frac{\eps}{2}\right)^{2}\frac{c}{6},\frac{\eps}{4}\}$.
\begin{proof} It is easy to see that $R(M(0))$ is closed. Let $( M_n)_n$ in $L(H)$ be such that $M(z)=\sum_{n=0}^\infty z^n  M_n$ for all $z\in B(0,\eps)$. By Lemma \ref{Le: CIF}, we have $\Abs{M'(0)}\leqq \frac{2}{\eps}\Abs{M}_\infty$ and, for all $0<\delta \leqq\eps/4$ and $z\in B(0,\delta)$, $\Abs{\sum_{n=2}^\infty z^{n-1}  M_n}\leqq 2\delta \left(\frac{2}{\eps}\right)^{2} \Abs{M}_\infty $. For $\nu\geqq \max\{\nu_1,{\hat\delta}^{-1}\}$, $z\in B(1/(2\nu),1/(2\nu))$, $\phi=(\phi_1, \phi_2) \in R(M(0))\oplus N(M(0))$ and $\eta>0$,
\begin{align*}
& \langle \phi,\Re z^{-1} M(z)\phi\rangle \\&= \left(\Re z^{-1}\right)\langle \phi_1,M(0)\phi_1\rangle + \langle  \phi,\Re M'(0)\phi\rangle +\Re\langle  \phi,\sum_{n=2}^\infty z^{n-1} M_n\phi\rangle\\ 
 &\geqq \nu d \abs{\phi_1}^2+c\abs{\phi_2}^2-2\Abs{M'(0)}\abs{\phi_1}\abs{\phi_2}-\Abs{M'(0)}\abs{\phi_1}^2-2\hat\delta \Abs{M}_\infty \left(\frac{2}{\eps}\right)^{2}\abs{\phi}^2\\
 &\geqq \left(\nu d - \eta \Abs{M'(0)}^2 -\Abs{M'(0)}\right)\abs{\phi_1}^2+\left(c-\frac{1}{\eta}\right)\abs{\phi_2}^2 - \frac{c}{3}\abs{\phi}^2\\
 &\geqq\left(\nu d - \eta \left(\Abs{M}_\infty\frac{2}{\eps}\right)^2 -\Abs{M}_\infty\frac{2}{\eps}-\frac{c}{3}\right)\abs{\phi_1}^2+\left(\frac{2c}{3}-\frac{1}{\eta}\right)\abs{\phi_2}^2.
\end{align*}
If $\eta=3/c$, using $\nu>\nu_1$, we obtain
\[
 \langle \Re z^{-1} M(z)\phi,\phi\rangle \geqq \left(\nu d - \frac{3}{c} \left(\Abs{M}_\infty\frac{2}{\eps}\right)^2 -\Abs{M}_\infty\frac{2}{\eps}-\frac{c}{3}\right)\abs{\phi_1}^2+\frac{c}{3}\abs{\phi_2}^2
\geqq \frac{c}{3}\abs{\phi}^2.\qedhere
\]
\end{proof}
\end{rems}

This completes the general discussion on $0$-analytic material laws. In the following we focus on material laws, which satisfy the following assumption.
\begin{Ass}\label{Ass:1} Assume there exist Hilbert spaces $H_1, H_2$ and constants $\eps,c>0$, $0<r<\eps/2$ with \[M\in \s H^\infty(B(0,\eps);L(H_1\oplus H_2))\cap \s H^{\infty,c}(B(r,r);L(H_1\oplus H_2))\] and $R(M(0))\subseteqq H_1\oplus H_2$ closed. 
\end{Ass}

Before we turn to  similarity transformations on the material law, we study some properties of a material law satisfying Assumption \ref{Ass:1}. These properties are stated in the next theorem for which we need the following elementary prerequisite.
\begin{Le}\label{Le: Off=0} Let $H_1, H_2$ be Hilbert spaces. Assume that 
\[
  \begin{pmatrix} M_{11}& M_{12} \\ M_{21} & M_{22}\end{pmatrix} \in L(H_1\oplus H_2)
\]
is selfadjoint and positive definite. Then $M_{12}=M_{21}^*$ and if $M_{22}=0$ then $M_{12}=0$.
 \end{Le}
\begin{proof}
 It is easy to see that $M_{11}=M_{11}^*$ and $M_{22}=M_{22}^*$ and thus 
\[
  \begin{pmatrix} 0& M_{12} \\ M_{21} & 0\end{pmatrix} =\begin{pmatrix} M_{11}& M_{12} \\ M_{21} & M_{22}\end{pmatrix} -\begin{pmatrix} M_{11}& 0 \\ 0 & M_{22}\end{pmatrix}
\]
is selfadjoint. Assume now that $M_{22}=0$. If $M_{12}=M_{21}^*\neq 0$, then there exists $(\phi_1,\phi_2)\in H_1\oplus H_2$ such that $\Re \langle M_{12}\phi_2,\phi_1\rangle =\langle M_{12}\phi_2,\phi_1\rangle +\langle M_{21}\phi_1,\phi_2\rangle < 0$. For $\alpha>0$ we deduce that
\[
 0\leqq \left\langle \begin{pmatrix} M_{11}& M_{12} \\ M_{21} & 0\end{pmatrix} \begin{pmatrix}
                                                                                \phi_1 \\ \alpha \phi_2
                                                                               \end{pmatrix}, \begin{pmatrix}
                                                                                \phi_1 \\ \alpha \phi_2
                                                                               \end{pmatrix}\right\rangle
= \langle M_{11}\phi_1,\phi_1\rangle +\alpha \left(\langle M_{12}\phi_2,\phi_1\rangle +\langle M_{21}\phi_1,\phi_2\rangle\right),
\]
which yields a contradiction if $\alpha$ is chosen large enough.
\end{proof}
In the following, for Hilbert spaces $H_1,H_2$ we denote the canonical orthogonal projection $H_1\oplus H_2\to H_j$ onto the $j$th coordinate by $\pi_j$, $j\in\{1,2\}$.
\begin{Sa}\label{Sa:2times2yields4times4} Let $M$ satisfy Assumption \ref{Ass:1}. Using the notation from Assumption \ref{Ass:1}, we define 
\[
   G_1\coloneqq  R(\pi_1 M(0)\pi_1^*),\ G_2\coloneqq  N(\pi_1 M(0)\pi_1^*),\ G_3\coloneqq  R(\pi_2 M(0)\pi_2^*),\ G_4\coloneqq  N(\pi_2 M(0)\pi_2^*).
\]
Then $M$ has the following form:
\begin{multline*}
    M = \left( z\mapsto \begin{pmatrix} M^{(0)}_{11} & 0 & M^{(0)}_{13} &0 \\ 0 & 0 & 0&0 \\ M^{(0)}_{31} & 0 & M^{(0)}_{33}&0 \\
  0&0&0&0 \end{pmatrix}+z\begin{pmatrix} M^{(1)}_{11}(z) & M^{(1)}_{12}(z) & M^{(1)}_{13}(z) & M^{(1)}_{14}(z) \\ M^{(1)}_{21}(z) & M^{(1)}_{22}(z) & M^{(1)}_{23}(z) & M^{(1)}_{24}(z) \\ M^{(1)}_{31}(z) & M^{(1)}_{32}(z) & M^{(1)}_{33}(z) & M^{(1)}_{34}(z) \\
                M^{(1)}_{41}(z) & M^{(1)}_{42}(z) & M^{(1)}_{43}(z) & M^{(1)}_{44}(z)
\end{pmatrix}\right) \\ \in \s H^{\infty}\left(B(0,\eps);L\left(\bigoplus_{j=1}^4 G_j\right)\right),
\end{multline*}
where for $j,k\in \{1,2,3,4\}$ we have $M^{(1)}_{jk}\in \s H^\infty(B(0,\eps);L(G_k,G_j))$ and if $j,k\in \{1,3\}$ we have ${M^{(0)}_{kj}}^*=M^{(0)}_{jk} \in L(G_k,G_j)$. Moreover, there is $d>0$ such that $M^{(0)}_{jj}\geqq d$ for $j\in\{1,3\}$.
\end{Sa}
\begin{proof}
 By Proposition \ref{Prop: M0+M1}, we know that $M(0)$ is selfadjoint and strictly positive definite on its range. Thus, \[M^{(0)}_{jj}=\left(\pi_j^* M(0)\pi_j : G_j \to G_j\right)\] is selfadjoint and strictly positive definite and therefore $H_{\frac{1}{2}j+\frac{1}{2}} = G_{j}\oplus G_{j+1}$ for $j\in\{1,3\}$. We denote by $\rho_j:H_1\oplus H_2\to G_j$ the orthogonal projections onto $G_j$ and define $M_{jk}^{(0)} \coloneqq  \rho_j M(0)\rho_k^*$, $M_{jk}^{(1)}\coloneqq \rho_j\left(M-M(0)\right)\rho_k^*$ for all $j,k\in \{1,2,3,4\}$. Hence, 
\[
   M=\left( z\mapsto \begin{pmatrix} M^{(0)}_{11} & 0 & M^{(0)}_{13} & M^{(0)}_{14}  \\ 0 & 0 & M^{(0)}_{23}&M^{(0)}_{24} \\ M^{(0)}_{31} & M^{(0)}_{32} & M^{(0)}_{33}&0 \\
  M^{(0)}_{41} &M^{(0)}_{42}&0&0 \end{pmatrix}+z\begin{pmatrix} M^{(1)}_{11}(z) & M^{(1)}_{12}(z) & M^{(1)}_{13}(z) & M^{(1)}_{14}(z) \\ M^{(1)}_{21}(z) & M^{(1)}_{22}(z) & M^{(1)}_{23}(z) & M^{(1)}_{24}(z) \\ M^{(1)}_{31}(z) & M^{(1)}_{32}(z) & M^{(1)}_{33}(z) & M^{(1)}_{34}(z) \\
                M^{(1)}_{41}(z) & M^{(1)}_{42}(z) & M^{(1)}_{43}(z) & M^{(1)}_{44}(z)
\end{pmatrix}\right).
\]
As $M(0)$ is selfadjoint, Lemma \ref{Le: Off=0} shows that ${M^{(0)}_{kj}}^*=M^{(0)}_{jk} \in L(G_k,G_j)$ for all $k,j \in\{1,2,3,4\}$. Since $M(0)$ is positive definite, it follows that the block operator matrices 
\[
 \begin{pmatrix} M^{(0)}_{11} & 0 & 0 & M^{(0)}_{14}  \\ 0 & 0 & 0&0 \\ 0 & 0 & 0&0 \\
  M^{(0)}_{41} &0&0&0 \end{pmatrix},\ \begin{pmatrix} 0 & 0 & 0 & 0  \\ 0 & 0 & M^{(0)}_{23}&0 \\ 0 & M^{(0)}_{32} & M^{(0)}_{33}&0 \\
 0&0&0&0 \end{pmatrix},\ \begin{pmatrix} 0 & 0 & 0 & 0  \\ 0 & 0 & 0&M^{(0)}_{24} \\ 0& 0 & 0&0 \\
  0 &M^{(0)}_{42}&0&0 \end{pmatrix}
\]
are positive definite as well. Thus, by Lemma \ref{Le: Off=0}, we deduce that $M^{(0)}_{14}={M^{(0)}_{41}}^*=0$, $M^{(0)}_{23}={M^{(0)}_{32}}^*=0$ and $M^{(0)}_{24}={M^{(0)}_{42}}^*=0$.
\end{proof}
For the next theorem we note that, for Hilbert spaces $H_1,H_2$ and $B\in L(H_2,H_1)$, we have
\begin{equation}\label{eq:ineq}
   \begin{pmatrix} 1 & B \\ 0 & 1 \end{pmatrix}^{-1}=\begin{pmatrix} 1 & -B \\ 0 & 1 \end{pmatrix} \text{ and }\Abs{\begin{pmatrix} 1 & B \\ 0 & 1 \end{pmatrix}^{-1}}\leqq \sqrt{1+\Abs{B}+\Abs{B}^2}.
\end{equation} Moreover, we need the following lemmas:
\begin{Le}\label{Le: PAPQAP_nice} Let $H$ be a Hilbert space. Let $T\in L(H)$ be continuously invertible, $A,B\in L(H)$. If $A=T^*BT$, i.e., $A$ and $B$ are \emph{similar}, then $\Re A = T^* \Re B T$ and if in addition $\Re B\geqq c>0$ then $\Re A\geqq \frac{c}{\Abs{T^{-1}}^2}$.
\end{Le}
\begin{proof} We have $2\Re A = A^*+A = T^*B^*T+T^*BT=2 T^*\Re B T$. Assume that $\Re B\geqq c$ for some $c>0$. Then, for $\phi\in H$, \[\langle \Re A \phi,\phi\rangle = \langle \Re B T\phi,T\phi\rangle \geqq c\langle T\phi,T\phi\rangle\geqq \frac{c}{\Abs{T^{-1}}^2}\langle\phi,\phi\rangle.\qedhere\] 
\end{proof}


\begin{Le}\label{Le:sim_gen_hsf} Let $M$ satisfy Assumption \ref{Ass:1}. Using the notation from Assumption \ref{Ass:1}, we define 
\[
 G_1\coloneqq  R(\pi_1 M(0)\pi_1^*),\ G_2\coloneqq  N(\pi_1 M(0)\pi_1^*),\ G_3\coloneqq  R(\pi_2 M(0)\pi_2^*),\ G_4\coloneqq  N(\pi_2 M(0)\pi_2^*).
\] 
 Let $N^{(0)}_{13}\in L(G_3,G_1)$, $N^{(0)}_{14}\in L(G_4,G_1)$, $N^{(0)}_{41}\in L(G_1,G_4)$, $N^{(0)}_{24}\in L(G_4,G_2)$, $N_{1}^{(1)}\in \s H^{\infty}(B(0,\eps);L(G_3\oplus G_4,G_1\oplus G_2))$,  $N_{1'}^{(1)}\in \s H^{\infty}(B(0,\eps);L(G_1\oplus G_2,G_3\oplus G_4))$ and
\begin{align*}
 N_{1}&\coloneqq  \left(z\mapsto \begin{pmatrix} N^{(0)}_{13} & N^{(0)}_{14} \\ 0  & N^{(0)}_{24} \end {pmatrix}+zN_{1}^{(1)}(z)\right)\\
 N_{1'}&\coloneqq   \left(z\mapsto \begin{pmatrix} {N^{(0)}_{13}}^* & 0 \\ {N^{(0)}_{41}}  & {N^{(0)}_{24}}^* \end {pmatrix}+zN_{1'}^{(1)}(z)\right).
\end{align*}
Then,
\[
   \s M\coloneqq \left( z\mapsto \begin{pmatrix} 1 & N_1(z) \\0&1 \end{pmatrix} M(z) \begin{pmatrix} 1 & 0 \\N_{1'}(z)&1 \end{pmatrix}\right) \\ \in \s H^\infty(B(0,\eps);L(H_1\oplus H_2))
\]
and $R(M(0))=R(\s M(0))$, $\s M(0)\geqq d'$ on its range and $\Re \s M'(0)\geqq c'$ on the nullspace of $\s M(0)$, where
\[
  d'\coloneqq  d\left( \sqrt{1+\Abs{N^{(0)}_{13}}+\Abs{N^{(0)}_{13}}^2}\right)^{-2}\text{ and }
  c'\coloneqq  c\left( \sqrt{1+\Abs{N_{24}^{(0)}}+\Abs{N_{24}^{(0)}}^2}\right)^{-2},
\]
with $d>0$ being the constant of positive definiteness of $M(0)$ on its range from Theorem \ref{Sa:2times2yields4times4}.
\end{Le}
\begin{proof}
 Using the representation of $M$ given in Theorem \ref{Sa:2times2yields4times4}, we compute $\s M(0)$:
\begin{align*}
   \s M(0)& = \begin{pmatrix} 1 & \begin{pmatrix} N^{(0)}_{13} & N^{(0)}_{14} \\ 0  & N^{(0)}_{24} \end {pmatrix} \\0&1 \end{pmatrix} \begin{pmatrix} M^{(0)}_{11} & 0 & M^{(0)}_{13} &0 \\ 0 & 0 & 0&0 \\ M^{(0)}_{31} & 0 & M^{(0)}_{33}&0 \\
  0&0&0&0 \end{pmatrix}\begin{pmatrix} 1 & 0 \\ \begin{pmatrix} {N^{(0)}_{13}}^* & 0 \\ N^{(0)}_{41}  & {N^{(0)}_{42}}^* \end {pmatrix}  &1 \end{pmatrix}\\
& = \begin{pmatrix} 1 & \begin{pmatrix} N^{(0)}_{13} & 0 \\ 0  & 0 \end {pmatrix}\\0&1 \end{pmatrix} \begin{pmatrix} M^{(0)}_{11} & 0 & M^{(0)}_{13} &0 \\ 0 & 0 & 0&0 \\ M^{(0)}_{31} & 0 & M^{(0)}_{33}&0 \\
  0&0&0&0 \end{pmatrix} \begin{pmatrix} 1 & 0 \\ \begin{pmatrix} {N^{(0)}_{13}}^* & 0 \\ 0 & 0 \end {pmatrix} &1 \end{pmatrix}.
\end{align*}Hence, $\s M(0)$ is similar to a positive definite operator. Moreover, the similarity transformation commutes with the projector 
\[P\coloneqq \begin{pmatrix}  1&0&0&0\\0&0&0&0\\0&0&1&0\\0&0&0&0  \end{pmatrix}.\]Thus, $M(0)$ and $\s M(0)$ have the same range and are both strictly positive on it. Indeed, $M(0)$ is strictly positive on $G_1\oplus \{0\} \oplus G_3 \oplus \{0\}$ and since the similarity transformation is a bijection on $G_1\oplus \{0\} \oplus G_3 \oplus \{0\}$, $\s M(0)$ is a bijection on $G_1 \oplus \{0\} \oplus G_3 \oplus \{0\}$ as well. In view of Lemma \ref{Le: PAPQAP_nice} and Inequality \eqref{eq:ineq}, we deduce $\s M(0)\geqq d'$ on its range. Next, consider $(1-P)\s M'(0) (1-P)$. For this purpose, we compute
\begin{align*}
 &(1-P)\begin{pmatrix} 1 & N_1(z) \\0&1 \end{pmatrix} \\
&= \begin{pmatrix}  0&0&0&0\\0&1&0&0\\0&0&0&0\\0&0&0&1  \end{pmatrix}\left( \begin{pmatrix} \begin{pmatrix} 1&0\\ 0&1\end{pmatrix} & \begin{pmatrix} N^{(0)}_{13} & N^{(0)}_{14} \\ 0  & N^{(0)}_{24} \end {pmatrix} \\ \begin{pmatrix} 0&0\\ 0&0\end{pmatrix} &\begin{pmatrix} 1&0 \\ 0&1\end{pmatrix} \end{pmatrix}+\begin{pmatrix} \begin{pmatrix} 0&0\\ 0&0\end{pmatrix} & z N_1^{(1)}(z) \\ \begin{pmatrix} 0&0\\ 0&0\end{pmatrix} &\begin{pmatrix} 0&0\\ 0&0\end{pmatrix}\end{pmatrix}\right)\\ 
  &= \begin{pmatrix} \begin{pmatrix} 0&0\\ 0&1\end{pmatrix} & \begin{pmatrix}0 & 0 \\ 0  & N^{(0)}_{24} \end {pmatrix}\\ \begin{pmatrix} 0&0\\ 0&0\end{pmatrix} &\begin{pmatrix} 0&0\\ 0&1\end{pmatrix} \end{pmatrix}+\begin{pmatrix} \begin{pmatrix} 0&0\\ 0&0\end{pmatrix} & z\begin{pmatrix} 0 & 0 \\ N^{(1)}_{1,23}(z) &N^{(1)}_{1,24}(z) \end {pmatrix} \\ \begin{pmatrix} 0&0\\ 0&0\end{pmatrix} &\begin{pmatrix} 0&0 \\ 0&0\end{pmatrix} \end{pmatrix}
\end{align*}
with suitable $N^{(1)}_{1,2k}\in \s H^\infty(B(0,\eps); L(G_k,G_2))$ $(k\in\{3,4\})$ and, similarly, we find
\begin{align*}
 \begin{pmatrix} 1 & 0 \\N_{1'}(z)&1 \end{pmatrix}(1-P)&=\begin{pmatrix} \begin{pmatrix} 0&0\\ 0&1\end{pmatrix} &  \begin{pmatrix} 0&0\\ 0&0\end{pmatrix} \\ \begin{pmatrix}0 & 0 \\ 0  & {N^{(0)}_{24}}^* \end {pmatrix} &\begin{pmatrix} 0&0\\ 0&1\end{pmatrix} \end{pmatrix}
  +\begin{pmatrix} \begin{pmatrix} 0&0\\ 0&0\end{pmatrix} &  \begin{pmatrix} 0&0\\ 0&0\end{pmatrix} \\ z\begin{pmatrix} 0 & N^{(1)}_{1',32}(z) \\ 0  &N^{(1)}_{1',42}(z) \end {pmatrix} &\begin{pmatrix} 0&0 \\ 0&0\end{pmatrix} \end{pmatrix},
\end{align*}
for suitable $N^{(1)}_{1',k2}\in \s H^\infty(B(0,\eps); L(G_2,G_k))$ $(k\in\{3,4\})$.
We obtain
{\footnotesize
\begin{align*}
 &(1-P)\s M'(0) (1-P)\\
 &=  \begin{pmatrix} \begin{pmatrix} 0&0\\ 0&1\end{pmatrix} & \begin{pmatrix}0 & 0 \\ 0  & N^{(0)}_{24} \end {pmatrix}\\ \begin{pmatrix} 0&0\\ 0&0\end{pmatrix} &\begin{pmatrix} 0&0\\ 0&1\end{pmatrix} \end{pmatrix}\begin{pmatrix} M^{(1)}_{11}(0) & M^{(1)}_{12}(0) & M^{(1)}_{13}(0) & M^{(1)}_{14}(0) \\ M^{(1)}_{21}(0) & M^{(1)}_{22}(0) & M^{(1)}_{23}(0) & M^{(1)}_{24}(0) \\ M^{(1)}_{31}(0) & M^{(1)}_{32}(0) & M^{(1)}_{33}(0) & M^{(1)}_{34}(0) \\
                M^{(1)}_{41}(0) & M^{(1)}_{42}(0) & M^{(1)}_{43}(0) & M^{(1)}_{44}(0)
\end{pmatrix}\begin{pmatrix} \begin{pmatrix} 0&0\\ 0&1\end{pmatrix} &  \begin{pmatrix} 0&0\\ 0&0\end{pmatrix} \\ \begin{pmatrix}0 & 0 \\ 0  & {N^{(0)}_{24}}^* \end {pmatrix} &\begin{pmatrix} 0&0\\ 0&1\end{pmatrix} \end{pmatrix}\\
&=  \begin{pmatrix} \begin{pmatrix} 0&0\\ 0&1\end{pmatrix} & \begin{pmatrix}0 & 0 \\ 0  & N^{(0)}_{24} \end {pmatrix}\\ \begin{pmatrix} 0&0\\ 0&0\end{pmatrix} &\begin{pmatrix} 0&0\\ 0&1\end{pmatrix} \end{pmatrix}\begin{pmatrix}0 & 0 & 0 & 0 \\ 0 & M^{(1)}_{22}(0) & 0 & M^{(1)}_{24}(0) \\ 0 & 0 & 0 & 0 \\
               0 & M^{(1)}_{42}(0) & 0 & M^{(1)}_{44}(0)
\end{pmatrix}\begin{pmatrix} \begin{pmatrix} 0&0\\ 0&1\end{pmatrix} &  \begin{pmatrix} 0&0\\ 0&0\end{pmatrix} \\ \begin{pmatrix}0 & 0 \\ 0  & {N^{(0)}_{24}}^* \end {pmatrix} &\begin{pmatrix} 0&0\\ 0&1\end{pmatrix} \end{pmatrix}
\end{align*}}Now, by Lemma \ref{Le: PAPQAP_nice} and Inequality \eqref{eq:ineq}, we see $\Re \s M'(0)\geqq c'$ on the nullspace of $\s M(0)$. 
\end{proof}

\begin{rems}\label{Le:asymp_exp_gen_hsf} Consider the following situation where we apply Lemma \ref{Le:sim_gen_hsf}. Let $G_j$ be a Hilbert space for $j\in\{1,2,3,4\}$.
 Let 
\begin{align*}
 N_{2}&\coloneqq  \left(z\mapsto \begin{pmatrix} N^{(0)}_{13} & 0 \\ 0  & 0 \end {pmatrix} + z \begin{pmatrix} N^{(1)}_{13}(z) & N^{(1)}_{14}(z) \\
                                                                                                          N^{(1)}_{23}(z) & N^{(1)}_{24}(z)
                                                                                          \end{pmatrix}\right)\in \s H^{\infty}(B(0,\eps);L(G_3\oplus G_4,G_1\oplus G_2))\\
 N_{2'}&\coloneqq  \left(z\mapsto \begin{pmatrix} N^{(0)}_{31} & 0 \\ 0  & 0 \end {pmatrix} + z \begin{pmatrix}  N^{(1)}_{31}(z) & N^{(1)}_{32}(z)\\
                                                                                                          N^{(1)}_{41}(z) & N^{(1)}_{42}(z)
                                                                                          \end{pmatrix}\right)\in \s H^{\infty}(B(0,\eps);L(G_1\oplus G_2,G_3\oplus G_4))
\end{align*}
and assume ${N^{(0)}_{13}}^*=N^{(0)}_{31}\in L(G_3,G_1)$. Moreover, let
\[
 N_3\coloneqq  \left(z\mapsto \begin{pmatrix} N^{(0)}_{33}+zN^{(1)}_{33}(z) & N_{34}^{(1)}(z) \\ N_{43}^{(1)}(z)  & z^{-1}N^{(0)}_{44}+N^{(1)}_{44}(z) \end {pmatrix} \right)
\]
with ${N^{(0)}_{33}}^*=N^{(0)}_{33}\in L(G_3)$, $N^{(0)}_{44}\in L(G_4)$, $N^{(1)}_{jk}\in \s H^\infty(B(0,\eps);L(G_k,G_j))$, $j,k\in \{3,4\}$. 
Then it is easy to see that $N_2(\cdot)N_3(\cdot) \in \s H^{\infty}(B(0,\eps);L(G_3\oplus G_4,G_1\oplus G_2))$ and $N_3(\cdot)N_{2'}(\cdot) \in \s H^{\infty}(B(0,\eps);L(G_1\oplus G_2,G_3\oplus G_4))$ and the following expansions hold
\begin{align*}
 N_2(z)N_3(z)&= \begin{pmatrix} N^{(0)}_{13}N^{(0)}_{33} & N^{(0)}_{13}N^{(1)}_{34}(0) \\ 0  & 0 \end {pmatrix}+\begin{pmatrix}0 & N^{(1)}_{14}(0)N^{(0)}_{44} \\ 0  & N^{(1)}_{24}(0)N^{(0)}_{44} \end {pmatrix}+O(z)
\end{align*}
and
\begin{align*}
 N_3(z)N_{2'}(z)&= \begin{pmatrix} N^{(0)}_{33}N^{(0)}_{31} & 0 \\ N^{(1)}_{43}(0)N^{(0)}_{31}  & 0 \end {pmatrix}+\begin{pmatrix}0 & 0 \\ N^{(0)}_{44}N^{(1)}_{41}(0)  &N^{(0)}_{44} N^{(1)}_{42}(0) \end {pmatrix}+O(z)
\end{align*}
 for $z\to 0$.
Now, let $M$ and the $G_j$'s be as in Lemma \ref{Le:sim_gen_hsf}. Assume the following compatibility condition
\[
   \left(N^{(0)}_{44}N^{(1)}_{42}(0)\right)^*=N^{(1)}_{24}(0)N^{(0)}_{44}.
\]
Then $N_1\coloneqq N_2N_3$ and $N_{1'}\coloneqq N_3N_{2'}$ satisfy the assumptions from Lemma \ref{Le:sim_gen_hsf}.
\end{rems}
We now turn to the analysis of inverses of material laws. Since we need to estimate the norm bounds of these inverses, we observe that, for a Hilbert space $H$ and a continuous linear operator $B\in L(H)$ satisfying $\Re B\geqq h$ for some $h>0$, $B^{-1}\in L(H)$. Moreover, using the Cauchy-Schwarz inequality, we deduce the estimate 
\begin{equation}\label{eq:alm_triv}
  \Abs{B^{-1}}\leqq 1/h.
\end{equation}
 Another consequence of $\Re B\geqq h$ is 
\begin{equation}\label{eq:alm_alm_triv}
 \Re B^{-1}\geqq h/(\Abs{B}^2). 
\end{equation}
 
\begin{Le}\label{Le: prototype}Let $H_1,H_2$ be Hilbert spaces, $d,c,\eps>0$. Let $L(H_1)\ni M_{11}^{(0)}={M_{11}^{(0)}}^*\geqq d$. Moreover, let $M_{jk}^{(1)}\in \s H^{\infty}(B(0,\eps); L(H_k,H_j))$ with $\Re M_{22}^{(1)}(0)\geqq c$. Define \[M_{I}\coloneqq  \left(B(0,\eps)\ni z\mapsto \begin{pmatrix} M^{(0)}_{11} & 0 \\ 0 & 0 \end {pmatrix} + z \begin{pmatrix} M^{(1)}_{11}(z) & M^{(1)}_{12}(z) \\
                                                                                                          M^{(1)}_{21}(z) & M^{(1)}_{22}(z)
                                                                                          \end{pmatrix}\right).
\]
Then there exists $\eps'>0$ depending on $\eps,\Abs{M_I}_\infty,c$ and $d$ such that, for $z\in B(0,\eps')\setminus \{0\}$,
{\footnotesize
\begin{multline*}
   M_{I}(z)^{-1} \\= \begin{pmatrix}
                        M_{121}(z) &\ -M_{121}(z)M_{12}^{(1)}(z)M_{22}^{(1)}(z)^{-1} \\ 
                       -M_{22}^{(1)}(z)^{-1}M_{21}^{(1)}(z)M_{121}(z) & \ M_{22}^{(1)}(z)^{-1}M_{21}^{(1)}(z)M_{121}(z)M_{12}^{(1)}(z)M_{22}^{(1)}(z)^{-1} + z^{-1}M_{22}^{(1)}(z)^{-1}  
                     \end{pmatrix},
\end{multline*}} where
\begin{align*}
   M_{121}(z)&\coloneqq  \left( M_{11}^{(0)}+ z \left(M_{11}^{(1)}(z)- M_{12}^{(1)}(z)M_{22}^{(1)}(z)^{-1}M_{21}^{(1)}(z)\right)\right)^{-1}\\
             &=\left(M_{11}^{(0)}\right)^{-1}+O(z)\text{ and }\\ 
  M_{22}^{(1)}(z)^{-1}&= M_{22}^{(1)}(0)^{-1}+ O(z)
\end{align*}
On the other hand, for $\hat M_{jk}\in \s H^\infty(B(0,\eps');L(H_{k},H_j))$ with $\Re \hat M_{22}^{(1)}(0)\geqq c$ and  
\[\hat M_{I}\coloneqq  \left(B(0,\eps')\setminus\{0\}\ni z\mapsto \begin{pmatrix} M^{(0)}_{11}+ z\hat M^{(1)}_{11}(z)& \hat M^{(1)}_{12}(z) \\ \hat M^{(1)}_{21}(z) & z^{-1}\hat M_{22}^{(1)}(z)\end {pmatrix}\right)
\] there exists $\eps''>0$ depending on $c,d,\Abs{\hat M_{jk}}_\infty(j,k\in\{1,2\})$ such that, for all $z\in B(0,\eps'')$,
{\footnotesize\begin{multline*}
  \hat M_{I}(z)^{-1}\\ = \begin{pmatrix}
                        \hat M_{121}(z) & -z\hat M_{121}(z)\hat M_{12}^{(1)}(z)\hat M_{22}^{(1)}(z)^{-1} \\ 
                       -z\hat M_{22}^{(1)}(z)^{-1}\hat M_{21}^{(1)}(z)\hat M_{121}(z) & z^2\hat M_{22}^{(1)}(z)^{-1}\hat M_{21}^{(1)}(z)\hat M_{121}(z)M_{12}^{(1)}(z)\hat M_{22}^{(1)}(z)^{-1} + z\hat M_{22}^{(1)}(z)^{-1}.  
                     \end{pmatrix},
\end{multline*}}
where
\begin{align*}
   \hat M_{121}(z)&\coloneqq  \left( M_{11}^{(0)}+ z\left(\hat M_{11}^{(1)}(z)- \hat M_{12}^{(1)}(z)z\hat M_{22}^{(1)}(z)^{-1}\hat M_{21}^{(1)}(z)\right)\right)^{-1}\\
             &=\left(M_{11}^{(0)}\right)^{-1}+O(z)\text{ and }
  \hat M_{22}^{(1)}(z)^{-1}&= \hat M_{22}^{(1)}(0)^{-1}+ O(z).
\end{align*}
In particular, $M_{I'}\in \s H^{\infty}(B(0,\eps'');L(H_1\oplus H_2))$.
\end{Le}
\begin{proof}
 The expressions for the inverses of $M_{I}$ and $M_{I'}$ can be verified immediately. The asymptotic expansions are straightforward applications of the Neumann series expansion. The respective convergence radii can be estimated in terms of $\eps,\Abs{M_I}_\infty,c$ and $d$ or $\eps,\Abs{\hat M_{jk}}_\infty(j,k\in\{1,2\}),c$ and $d$ by Lemma \ref{Le: CIF} and Inequality \eqref{eq:alm_triv}. 
\end{proof}

\begin{Sa}\label{Th:final} Let $H_1,H_2$ be separable Hilbert spaces, $c,d,\eps,r>0$. Let $(N_n)_n=\left(\left(\begin{smallmatrix}N_{11,n}&N_{12,n}\\N_{21,n}&N_{22,n}  \end{smallmatrix}\right)\right)_n$ be a bounded sequence in \[\s H^\infty(B(0,\eps);L(H_1\oplus H_2))\cap \s H^{\infty,c}(B(r,r);L(H_1\oplus H_2))\] and $M\in \s H^\infty(B(0,\eps);L(H_1\oplus H_2))\cap \s H^{\infty,c}(B(r,r);L(H_1\oplus H_2))$. Assume that for all $n\in\N$ we have $R(M(0))=R(N_n(0))$ and $M(0),N_n(0)\geqq d$ on $R(M(0))$. Denote by $q_j:H_j \to R(\pi_j^*) \cap N(M(0))$ $(j\in\{1,2\})$ the canonical ortho-projections. Assume for all $n\in \N$ the compatibility condition \[\left(\left(q_2N_{22,n}'(0)q_2^*\right)^{-1}q_2N_{21,n}'(0)q_1^{*}\right)^*=q_1N_{12,n}'(0)q_2^{*}\left(q_2N_{22,n}'(0)q_2^*\right)^{-1}.\]
Then there exist $\eps',r',c'>0$ depending on $\eps,c,r,\Abs{M}_\infty,\sup\{\Abs{N_n}_\infty;n\in\N\},d$ such that
\begin{align*}
   (N_{1,n})_n\coloneqq &\left(\left(z\mapsto  N_{12,n}(z)N_{22,n}(z)^{-1}\right)\right)_n \\
   (N_{1',n})_n\coloneqq &\left(\left(z\mapsto N_{22,n}(z)^{-1}N_{21,n}(z)\right)\right)_n
\end{align*}
are bounded in $\s H^\infty(B(0,\eps');L(H_2,H_1))$ and $ \s H^\infty(B(0,\eps');L(H_1,H_2))$, respectively. Moreover, denoting by $M_{1}\in \s H_{\textnormal{w}}^\infty(B(0,\eps');L(H_2,H_1))$ and $M_{1'}\in \s H_{\textnormal{w}}^\infty(B(0,\eps');L(H_1,H_2))$ the respective limits of $(N_{1,n_k})_k$ and $(N_{1',n_k})_k$ for a strictly monotone sequence of positive integers $(n_k)_k$, we have
\begin{multline*}
   \s M\coloneqq \left( z\mapsto \begin{pmatrix} 1 &\pm M_1(z) \\0&1 \end{pmatrix} M(z) \begin{pmatrix} 1 & 0 \\ \pm M_{1'}(z)&1 \end{pmatrix}\right) \\ \in \s H^\infty(B(0,\eps');L(H_1\oplus H_2))\cap \s H^{\infty,c'}(B(r',r');L(H_1\oplus H_2)),
\end{multline*}
and $R(\s M(0))=R(M(0))$. 
\end{Sa}
\begin{proof} In the following, we use Hilbert spaces $G_j$, $j\in \{1,2,3,4\}$ as in Theorem \ref{Sa:2times2yields4times4} and represent $N_n$ for $n\in\N$ using bounded operators $N_{jk,n}^{(0)}$, $j,k\in\{1,3\}$, and Hardy space functions $N_{jk,n}^{(1)}$,
 $j,k\in\{1,2,3,4\}$, as in Theorem \ref{Sa:2times2yields4times4}. From Lemma \ref{Le: prototype} we have an explicit expression for $N_{22,n}(z)^{-1}$, namely
\[
    N_{22,n}(z)^{-1} =\begin{pmatrix} \left(N_{33,n}^{(0)}\right)^{-1}+O(z) & O(1) \\ O(1) & z^{-1} N_{44,n}^{(1)}(0)^{-1}+O(1)\end{pmatrix} \text{ for }z\to 0.
\]
Moreover, we have an estimate for the radius of convergence $\eps'$ for the Neumann expansion involved in this expression in terms of $\sup\{\Abs{N_n}_\infty;n\in\N\}, d,c,\eps$. In particular, $z\mapsto N_{22,n}(z)^{-1}$ satisfies the assumptions on $N_3$ in Remark \ref{Le:asymp_exp_gen_hsf} (note that $\left(\left(N_{33,n}^{(0)}\right)^{-1}\right)^*=\left(N_{33,n}^{(0)}\right)^{-1}$, by Theorem \ref{Sa:2times2yields4times4}). Moreover, using Theorem \ref{Sa:2times2yields4times4}, we deduce that $N_{12,n}$ and $N_{21,n}$ satisfy the assumptions on $N_2$ and $N_{2'}$ in Remark \ref{Le:asymp_exp_gen_hsf}. Indeed, we have
\begin{align*}
 N_{12,n}(z)&\coloneqq  \left(z\mapsto \begin{pmatrix} N^{(0)}_{13,n} & 0 \\ 0  & 0 \end {pmatrix} + z \begin{pmatrix} N^{(1)}_{13,n}(z) & N^{(1)}_{14,n}(z) \\
                                                                                                          N^{(1)}_{23,n}(z) & N^{(1)}_{24,n}(z)
                                                                                          \end{pmatrix}\right)\\
&\quad\quad\quad\quad\quad\quad\in \s H^{\infty}(B(0,\eps);L(G_3\oplus G_4,G_1\oplus G_2))\\
 N_{21,n}(z)&\coloneqq  \left(z\mapsto \begin{pmatrix} N^{(0)}_{31,n} & 0 \\ 0  & 0 \end {pmatrix} + z \begin{pmatrix}  N^{(1)}_{31,n}(z) & N^{(1)}_{32,n}(z)\\
                                                                                                          N^{(1)}_{41,n}(z) & N^{(1)}_{42,n}(z)
                                                                                          \end{pmatrix}\right) \\
&\quad\quad\quad\quad\quad\quad \in \s H^{\infty}(B(0,\eps);L(G_1\oplus G_2,G_3\oplus G_4))
\end{align*}
with $\left( N^{(0)}_{13,n}\right)^*=N^{(0)}_{31,n}$ by Theorem \ref{Sa:2times2yields4times4}.

Moreover, Remark \ref{Le:asymp_exp_gen_hsf} shows that $N_{1,n}$ and $N_{1',n}$ satisfy the assumptions imposed on $N_1$ and $N_{1'}$ in Lemma \ref{Le:sim_gen_hsf}. More precisely, we have the expansions
\begin{align*}
   N_{1,n}(z)&=\begin{pmatrix}
             N_{13,n}^{(0)}(N_{33,n}^{(0)})^{-1} & \hat N_{14,n} \\ 0 & N_{24,n}^{(1)}(0) N_{44,n}^{(1)}(0)^{-1}
            \end{pmatrix}+O(z)\\
    N_{1',n}(z)&=\begin{pmatrix}
              (N_{33,n}^{(0)})^{-1}N_{31,n}^{(0)} & 0\\ \hat N_{41,n}  & N_{44,n}^{(1)}(0)^{-1} N_{42}^{(1)}(0)
             \end{pmatrix}+O(z)
\end{align*}for suitable continuous linear operators $\hat N_{14,n},\hat N_{41,n}$. We deduce that \[
\left(N_{13,n}^{(0)}(N_{33,n}^{(0)})^{-1}\right)^*=(N_{33,n}^{(0)})^{-1}N_{31,n}^{(0)}.\]
Moreover, the compatibility condition is precisely
\[
    \left( N_{24,n}^{(1)}(0)N_{44,n}^{(1)}(0)^{-1}\right)^* = N_{44,n}^{(1)}(0)^{-1} N_{42,n}^{(1)}(0). 
\]
Lemma \ref{Le: Conv_of_coeff} together with the fact that computing the adjoint is a continuous process in the weak operator topology ensures that $M_1$ and $M_{1'}$ satisfy the assumptions imposed on $N_1$ and $N_{1'}$ in Lemma \ref{Le:sim_gen_hsf}. To estimate the norm bounds of $M_1$ and $M_{1'}$ in terms of $d,c,\sup\{\Abs{N_n}_\infty;n\in\N\}$ and $\eps$, we use Lemma \ref{Le: CIF} and Lemma \ref{Le: prototype}. Hence with the help of Lemma \ref{Le:sim_gen_hsf}, we may estimate the constants of positive definiteness of $\s M(0)$ and $\Re \s M'(0)$  on $R(\s M(0))$ and $N(\s M(0))$, respectively, also in terms of $d,c,\sup\{\Abs{N_n}_\infty;n\in\N\}$ and $\eps$. Note that we also have $R(\s M(0))=R(M(0))$. Remark \ref{rem:conv} implies the remaining assertion.
\end{proof}

\section*{Acknowledgments}
The author is indebted to Ralph Chill, Des McGhee, Rainer Picard and Sascha Trostorff for carefully reading the manuscript. The author is also indebted to the anonymous referee for a careful reading of this manuscript and for numerous constructive comments. M.~W.~carried out this work with partial financial support of the EPSRC grant EP/L018802/2: ``Mathematical foundations of metamaterials: homogenisation, dissipation and operator theory''. This is gratefully acknowledged. 

\bibliographystyle{plain}

\end{document}